\documentclass[final,leqno,onefignum,onetabnum]{siamltex1213}
\usepackage{amsmath}
\usepackage{bm}
\usepackage{amsfonts}
\usepackage{amssymb}
\usepackage{mathrsfs}
\usepackage{graphicx}
\usepackage{mathtools}
\newcommand{\myalgsize}{\footnotesize}
\usepackage{algorithm,algcompatible}
\newtheorem{remark}{Remark}
\newtheorem{assumption}{Assumption}

\makeatletter
\let\save@mathaccent\mathaccent
\newcommand*\if@single[3]{%
  \setbox0\hbox{${\mathaccent"0362{#1}}^H$}%
  \setbox2\hbox{${\mathaccent"0362{\kern0pt#1}}^H$}%
  \ifdim\ht0=\ht2 #3\else #2\fi
  }
\newcommand*\rel@kern[1]{\kern#1\dimexpr\macc@kerna}
\newcommand*\widebar[1]{\@ifnextchar^{{\wide@bar{#1}{0}}}{\wide@bar{#1}{1}}}
\newcommand*\wide@bar[2]{\if@single{#1}{\wide@bar@{#1}{#2}{1}}{\wide@bar@{#1}{#2}{2}}}
\newcommand*\wide@bar@[3]{%
  \begingroup
  \def\mathaccent##1##2{%
    \let\mathaccent\save@mathaccent
    \if#32 \let\macc@nucleus\first@char \fi
    \setbox\z@\hbox{$\macc@style{\macc@nucleus}_{}$}%
    \setbox\tw@\hbox{$\macc@style{\macc@nucleus}{}_{}$}%
    \dimen@\wd\tw@
    \advance\dimen@-\wd\z@
    \divide\dimen@ 3
    \@tempdima\wd\tw@
    \advance\@tempdima-\scriptspace
    \divide\@tempdima 10
    \advance\dimen@-\@tempdima
    \ifdim\dimen@>\z@ \dimen@0pt\fi
    \rel@kern{0.6}\kern-\dimen@
    \if#31
      \overline{\rel@kern{-0.6}\kern\dimen@\macc@nucleus\rel@kern{0.4}\kern\dimen@}%
      \advance\dimen@0.4\dimexpr\macc@kerna
      \let\final@kern#2%
      \ifdim\dimen@<\z@ \let\final@kern1\fi
      \if\final@kern1 \kern-\dimen@\fi
    \else
      \overline{\rel@kern{-0.6}\kern\dimen@#1}%
    \fi
  }%
  \macc@depth\@ne
  \let\math@bgroup\@empty \let\math@egroup\macc@set@skewchar
  \mathsurround\z@ \frozen@everymath{\mathgroup\macc@group\relax}%
  \macc@set@skewchar\relax
  \let\mathaccentV\macc@nested@a
  \if#31
    \macc@nested@a\relax111{#1}%
  \else
    \def\gobble@till@marker##1\endmarker{}%
    \futurelet\first@char\gobble@till@marker#1\endmarker
    \ifcat\noexpand\first@char A\else
      \def\first@char{}%
    \fi
    \macc@nested@a\relax111{\first@char}%
  \fi
  \endgroup
}
\makeatother
\DeclareMathOperator*{\argmin}{\arg\!\min}

\newcommand{\real}{{\rm{I\hspace{-.75mm}R}}}

\renewcommand*{~}{\relax\ifmmode\sim\else\nobreakspace{}\fi}%GET TILDE FOR "IS DISTRIBUTED AS"

\newcommand{\Fbar}{\bar{F}}

\newcommand{\alphahat}{\hat{\alpha}}

\newcommand{\rhohat}{\hat{\rho}}
\newcommand{\sigmahat}{\hat{\sigma}}

\newcommand{\Ntilde}{\widetilde{N}}

\newcommand{\Deltatilde}{\tilde{\Delta}}

\newcommand{\BFe}{\bm{e}}

\newcommand{\BFs}{\bm{s}}

\newcommand{\BFx}{\bm{x}}
\newcommand{\BFy}{\bm{y}}
\newcommand{\BFz}{\bm{z}}
\newcommand{\BFG}{\bm{G}}

\newcommand{\BFX}{\bm{X}}
\newcommand{\BFS}{\bm{S}}
\newcommand{\BFY}{\bm{Y}}

\newcommand{\BFalpha}{\bm{\alpha}}

\newcommand{\BFXtilde}{\tilde{\BFX}}

\newcommand{\mcB}{\mathcal{B}}

\newcommand{\mcD}{\mathcal{D}}

\newcommand{\mcF}{\mathcal{F}}

\newcommand{\mcK}{\mathcal{K}}

\newcommand{\mcO}{\mathcal{O}}

\newcommand{\mcY}{\mathcal{Y}}

 \newcommand{\mcYtilde}{\tilde{\mcY}}

\newcommand{\mbE}{\mathbb{E}}

\newcommand{\mbN}{\mathbb{N}}

\newcommand{\mbP}{\mathbb{P}}

\newcommand{\defn}{\mathrel{\mathop:}=}

\title{%
  ASTRO-DF: A Class of Adaptive Sampling \\ 
Trust-Region Algorithms for Derivative-Free Stochastic Optimization
}
\author{Sara Shashaani \thanks{School of Industrial Engineering, Purdue University (email: \email{sshashaa@purdue.edu})} \and Fatemeh S. Hashemi \thanks{Department of Industrial and Systems Engineering, Virginia Tech (email: \email{fatemeh@vt.edu})} \and Raghu Pasupathy \thanks{Department of Statistics, Purdue University (email: \email{pasupath@purdue.edu})}
}

\date{}
\begin{document}
\maketitle
%\slugger{siopt}{xxxx}{xx}{x}{x--x}%slugger should be set to mms, siap, sicomp, sicon, sidma, sima, simax, sinum, siopt, sisc, or sirev

\begin{abstract}
We consider unconstrained optimization problems where only ``stochastic" estimates of the objective function are observable as replicates from a Monte Carlo oracle. The Monte Carlo oracle is assumed to provide no direct observations of the function gradient. We present ASTRO-DF --- a class of derivative-free trust-region algorithms, where a stochastic local interpolation model is constructed, optimized, and updated iteratively. Function estimation and model construction within ASTRO-DF is \emph{adaptive} in the sense that the extent of Monte Carlo sampling is determined by continuously monitoring and balancing metrics of sampling error (or variance) and structural error (or model bias) within ASTRO-DF. Such balancing of errors is designed to ensure that Monte Carlo effort within ASTRO-DF is sensitive to algorithm trajectory, sampling more whenever an iterate is inferred to be close to a critical point and less when far away. We demonstrate the almost-sure convergence of ASTRO-DF's iterates to a first-order critical point when using linear or quadratic stochastic interpolation models. The question of using more complicated models, e.g., regression or stochastic kriging, in combination with adaptive sampling is worth further investigation and will benefit from the methods of proof presented here. We speculate that ASTRO-DF's iterates achieve the canonical Monte Carlo convergence rate, although a proof remains elusive. 
\end{abstract}


%\begin{AMS}\end{AMS}

\pagestyle{myheadings}
\thispagestyle{plain}
\markboth{Shashaani, Hashemi, and Pasupathy}{ASTRO-DF}

\section{INTRODUCTION}
\label{sec:intro}
We consider unconstrained stochastic optimization (SO) problems, that is, optimization problems in continuous space where the objective function(s) can only be expressed implicitly via a Monte Carlo oracle. The Monte Carlo oracle is assumed to not provide any direct observations of the function derivatives. 

SO has recently gathered attention due to its versatile formulation, allowing the user to specify functions involved in an optimization problem implicitly, e.g., through a stochastic simulation. As a result, SO allows virtually any level of problem complexity to be embedded, albeit at the possible price of a computationally burdensome and slow Monte Carlo oracle. SO has seen wide recent adoption --- see, for example, applications in telecommunication networks \cite{hou2014applied}, traffic control \cite{osorio2009surrogate}, epidemic forecasting \cite{nsoesie2013simulation} and health care \cite{alagoz1}. Recent editions of the Winter Simulation Conference (\texttt{www.informs-sim.org}) have dedicated an entire track to the SO problem and its various flavors. For a library of SO problems, see~\texttt{www.simopt.org} and \cite{pasupathy2006testbed,pasupathy_testbed2011}.

\subsection{Problem Statement}
The SO problem we consider is formally stated as follows: 
\begin{equation} \mbox{Problem } P: \mbox{minimize } f\left(\BFx\right) \mbox{ subject to } \BFx\in \real^d,\end{equation}
where $f:\real^{d}\rightarrow\real$ is bounded from below and has Lipschitz continuous gradients. Furthermore, the function $f(\BFx) = \mathbb{E}[F(\BFx)]$ is the expectation of a random function $F(\BFx)$ that is observable through Monte Carlo. This means, for instance, that one can generate $n$ identically distributed samples or replicates $F_i(\BFx), j =1, 2, \ldots,n$ of $F(\BFx)$ by ``executing" the Monte Carlo simulation $n$ times at the point $\BFx$. This leads to the estimator $\Fbar\left(\BFx,n\right)=n^{-1}\sum_{j=1}^{n}F_{j}\left(\BFx\right)$ having standard error estimated as $\sigmahat_F\left(\BFx,n\right)/\sqrt{n}$ where $\sigmahat^{2}_F\left(\BFx,n\right)= n^{-1}\sum_{j=1}^{n}\left(F_{j}\left(\BFx\right)-\Fbar\left(\BFx,n\right)\right)^{2}$. We assume that no direct observations of the gradient $\nabla f(\cdot)$ , e.g., through IPA~\cite[p. 214]{asmussen2007stochastic}, are available through the Monte Carlo oracle. This means that methods seeking a gradient estimate need to resort to indirect methods such as as finite differencing~\cite[p. 209]{asmussen2007stochastic}, leading to biased estimators.

An algorithm for solving the above problem will be evaluated based on its ability to return a (random) sequence of iterates $\{\BFX_{k}\}$ converging in some rigorously defined probabilistic metric to a first- or second-order critical point of the function $f$. Thus, each ``run" of a solution algorithm will return a random sequence of iterates $\{\BFX_{k}\}$, and SO algorithms that return iterate sequences $\{\BFX_{k}\}$ guaranteed to converge to a critical point with probability one  will be called \emph{consistent}. %The phrase ``canonical rate" is sometimes used to refer to the fastest achievable convergence rate under generic Monte Carlo sampling. In the current context, we will say that the random sequence of iterates $\{\BFX_{k}\}$ achieves the canonical Monte Carlo rate if $\sqrt{W_k} \, \left\|\nabla f(\BFX_{k}) \right\|= \mcO_p(1)$, where $W_k$ is the total Monte Carlo effort expended after $k$ iterations.

\subsection{Complications}
The presence of a Monte Carlo oracle lends flexibility to the SO problem formulation, but it also brings with it a simply-stated complication: the lack of uniform deterministic error guarantees. Specifically, suppose $f\left(\BFx,n\right)$ is the Monte Carlo estimate of the unknown desired function value $f(\BFx)$ at the point $\BFx$, and $n$ represents the extent of Monte Carlo effort. Then, simple probability arguments reveal that deterministic guarantees of the sort $|f\left(\BFx,n\right) - f\left(\BFx\right)| \leq \epsilon, \epsilon>0$ do not hold irrespective of the size of $n$; instead, one has to be content with probabilistic precision guarantees of the form $\mbP\{| f\left(\BFx,n\right) - f\left(\BFx\right)| > \epsilon\} \leq \alpha$ for $n \geq n_0(\alpha)$. The analogous situation for function derivative estimation using Monte Carlo is worse. If the derivative estimate $\hat{\nabla} f(\BFx)\defn (\hat{\nabla}_1 f(\BFx), \hat{\nabla}_2 f(\BFx), \ldots, \hat{\nabla}_q f(\BFx))$ is constructed using a central-difference approximation as $$\hat{\nabla}_i f(\BFx) = (2c_n)^{-1}(f(\BFx+ c_n\BFe_{i}, n) - f(\BFx - c_n\BFe_{i}, n)), i = 1,2,\ldots,q,$$ then, as in the function estimation context, no uniform guarantees on the accuracy of $\hat{\nabla} f(\BFx)$ are available in general. Furthermore, the rate at which $\hat{\nabla} f(\BFx)$ converges to 
$\nabla f(\BFx)$ depends crucially on the choice of $\{c_n\}$, with the best possible rate $\mcO(n^{-1/3})$ under generic Monte Carlo sampling being much slower than the corresponding $\mcO(n^{-1/2})$ rate for function estimation. (See~\cite{asmussen2007stochastic} for this and related results.) Most importantly, implementing such finite-difference derivative estimates within an SO algorithm is well recognized to be a delicate issue, easily causing instabilities. In any event, the lack of uniform deterministic guarantees in the SO context means that estimation error inevitably accumulates across iterations of an algorithm, thereby threatening convergence guarantees of the resulting iterates. Algorithms for solving SO have to somehow contend with such potential non-convergence due to mischance, either through the introduction of gain sequences as in stochastic approximation~\cite{kushner2} or through appropriate sampling as in sample average approximation or retrospective approximation~\cite{kimpashen14,pas10}.

A second complication within SO, but one that it partially shares with black-box deterministic optimization contexts, is the lack of information about function structure. Structural properties such as convexity, uni-modality, and differentiability, if known to be present, can be exploited when designing optimization algorithms. Such properties, when appropriate, are usually assumed within the deterministic context, and an appropriate solution algorithm devised. In SO, however, structural assumptions about the underlying true objective and constraint function, even if correct, may not provide as much leverage during algorithm development. This is because, due to the presence of stochastic error, the true objective and constraint functions are never directly observed; and, making structural assumptions about their observed sample-paths is far more suspect. 

\begin{remark} Another aspect that is unique to SO is noteworthy. Monte Carlo calls are typically the most compute-intensive operations within SO contexts. And, depending on the nature of the SO algorithm, different number of Monte Carlo calls may be expended across iterations, e.g., constant in Stochastic Approximation (SA)~\cite{kushner2}, varying but predetermined in Retrospective Approximation (RA)~\cite{pas10}, or random in Sampling Controlled Stochastic Recursion (SCSR)~\cite{pasglyghohas14}. This means that the elemental measure of effort in SO --- the number of Monte Carlo oracle calls --- may not have a simple relationship with the notion of ``iterations" defined within the specific SO algorithm, forcing a need for more careful book-keeping. This is why iterative SO algorithms are well-advised to measure convergence and convergence rates not in terms of the number of iterations, but rather in terms of the total number of Monte Carlo calls. \end{remark} %A more subtle issue is that the stochastic element within SO imposes a certain upper bound on the maximum achievable convergence rate (also known as \emph{canonical rate}~\cite{asmussen2007stochastic}) of SO algorithms.\end{remark}

\subsection{ASTRO-DF and Overview of Contribution}

Our particular focus in this research is that of developing a class of algorithms for solving low to moderate dimensional SO problems that have no readily discernible structure. We are inspired by the analogous problem in the deterministic context that has spurred the development of a special and arguably very useful class  of optimization methods called model-based trust-region derivative-free (TRO-DF) algorithms~\cite{conschvic2009,powell2002uobyqa,conn2009global,bandeira2014convergence}. TRO-DF algorithms are typified by two aspects: (i) they eschew the direct computation and use of derivatives for searching, and instead rely on constructed models of guaranteed accuracy in specified ``trust-regions"; (ii) the algorithmic search evolves by repeatedly constructing and optimizing a local model within a dynamic trust-region, explicitly restricting the distance between the successive iterates returned by the algorithm. The aspect in (i) is particularly suited for adaptation to SO contexts where direct derivative estimation can be delicate and unstable, requiring careful choice of step-sizes~\cite{asmussen2007stochastic}; the aspect in (i) also aids efficiency because models constructed in previous iterations can be re-used with some updating, and no effort is expended for explicit estimation of derivatives. The aspect in (ii) runs counter to efficiency, but is designed to reduce variance in the algorithm's iterates, through steps that are more circumspect. 

We construct a family of adaptive sampling trust-region optimization derivative-free  (ASTRO-DF) algorithms for the SO context. In their most rudimentary form, ASTRO-DF algorithms follow a familiar idea for iteratively estimating the first and second order critical points of a function. Given a current random iterate $\BFX_{k}$ that approximates the first-order critical point of interest, ASTRO-DF constructs a tractable ``local" stochastic model using Monte Carlo observations of the objective function at carefully chosen points around $\BFX_{k}$. The constructed model is then optimized within the local region in which it is constructed to obtain a candidate solution $\BFXtilde_{k+1}$. Next, the objective function is \emph{observed} (using Monte Carlo) at $\BFXtilde_{k+1}$ and compared against the value \emph{predicted} by the model at $\BFXtilde_{k+1}$. If the observed decrease in function values from $\BFX_{k}$ to $\BFXtilde_{k+1}$ exceeds the decrease predicted by the constructed model in a certain stochastic sense, the candidate $\BFXtilde_{k+1}$ is accepted as the next iterate $\BFX_{k+1}$. As a vote of confidence on the constructed model, the trust-region radius is then expanded by a factor. Otherwise, that is, if the predicted decrease is much lower than the observed decrease (again, in a certain precise sense), the candidate $\BFXtilde_{k+1}$ is rejected, the trust-region radius is shrunk, and the local model is updated in an attempt to improve accuracy. This iterative process then repeats to produce a random sequence of iterates $\{\BFX_{k}\}$ that is realized in each run of ASTRO-DF. 

\begin{remark} Throughout this paper, we use the term ``sampling" to refer to the act of obtaining replicates using multiple runs of the Monte Carlo oracle at a fixed point. This is not to be confused with sampling design points in the search region. So, when we say that the sample size is $n$, we mean that $n$ amount of Monte Carlo effort was expended to obtain the function estimate \emph{at a fixed point}.\end{remark}

The above ideas for model construction, trust-region management, and candidate point acceptance say nothing about how much Monte Carlo effort to expend. Since all observations for function estimation and model construction are based on Monte Carlo, the resulting accuracy estimates are at best probabilistic, leading us to the question of how much to sample. Too little Monte Carlo effort threatens convergence due to accumulated stochastic and deterministic errors, and too much Monte Carlo sampling means reduced overall efficiency. Identifying the correct Monte Carlo sampling trade-off is more than a theoretical question, and answering it adequately entails more than broad prescriptions on sampling rates. To produce good implementations, sampling prescriptions ought to be automatic and specific to the problem at hand, which usually means relying on inference based on algorithm trajectory.    

To resolve the issue of how much to sample, we propose that a simple strategy called \emph{adaptive sampling} be incorporated within derivative-free trust-region algorithms in the SO context. Recognizing that the error in function and model estimation can be decomposed 
orthogonally into error due to sampling (or variance in the case of unbiased estimates) and error due to structure (or bias), adaptive sampling seeks to ensure that ``just adequate" Monte Carlo sampling is performed by balancing these errors. For example, when constructing a local model, Monte Carlo sampling in ASTRO-DF is adaptive in the sense that sampling continues until a certain continuously monitored metric of model quality exceeds a metric of sampling variability. A similar rule is employed when estimating the objective function at a point for purposes of candidate acceptance. We believe that such adaptive sampling paves the way for efficiency because it reacts to the observed algorithm trajectory and, as we shall see, keeps the different sources of error within the algorithm in lock-step. The resulting algorithm remains practical because of the simplicity of the proposed adaptive sampling rule --- sample until the estimated standard error falls below a certain specified power of the prevailing trust-region radius. 

\begin{remark} Adaptive sampling as an idea is not new and has been used with great success in other areas such as sequential confidence interval construction~\cite{chorob1965,ghomuksen97} and SO on finite spaces~\cite{shane3}.\end{remark}

Adaptive sampling, while invaluable as an implementation idea, introduces substantial complications when analyzing algorithm behavior. Akin to what happens during sequential sampling in the context of confidence interval construction~\cite{chorob1965,ghomuksen97}, the explicit dependence of the extent of Monte Carlo sampling on algorithm trajectory causes systematic early stopping and consequent bias in the function estimates obtained within ASTRO-DF. In other words, when using adaptive sampling, $\mathbb{E}[f(\BFx,n)] \neq f(\BFx)$ in general since the sample size $n$ is a stopping time~\cite[pp. 182]{durrett1} that will depend on $f(\BFx,n)$. Demonstrating that ASTRO-DF's iterates converge to a first-order critical point with probability one then entails demonstrating that the bias effects of adaptive sampling, especially when used within the derivative-free trust-region context, wear away asymptotically. We accomplish this by first generically characterizing a relationship between the moments of the adaptive sample size and the function estimates at stopping, and then showing that the errors induced due to model construction, algorithm recursion, and function estimation remain in lock-step (or comparable) throughout ASTRO-DF's evolution.   

We note that ASTRO-DF, as presented here, assumes that a stochastic linear or stochastic quadratic interpolation model is constructed during the model-construction step of the algorithm. While such models are reasonable and have seen wide use in the analogous deterministic context, other possibly more powerful model construction techniques such as regression or stochastic kriging~\cite{anknelsta2010} should be considered in place of interpolation models, especially alongside adaptive sampling. Ongoing research investigates this question and it is our belief that the proof techniques that we present in this paper will carry over, albeit with some changes.

\section{PRELIMINARIES}
\label{sec:pre}

In this section, we list notation, key definitions, and some basic results that will be used invoked throughout the rest of the document. 

\subsection{Notation and Convention}\label{sec:notcon}
We use bold font for vectors, script font for sets, lower case font for real numbers and upper case font for random variables. Hence $\{\BFX_{k}\}$ denotes a sequence of random vectors in $\real^{d}$, $\BFx=\left(x^1,x^2,\ldots,x^d\right)$ denotes a $d$-dimensional vector of real numbers, $\mcY:=\left\{\BFY_{1},\BFY_2,\ldots,\BFY_p\right\}$ denotes a set of $p$ real vectors, and $\mcY:=\left\{\BFY_{1},\BFY_2,\ldots,\BFY_p\right\}$ denotes a set of $p$ random vectors. The set $\mcB\left(\BFx;r\right)=\left\{\BFy\in\real^{d}:\ \left\| \BFy-\BFx\right\|\leq r \right\}$ is the closed ball of radius $r>0$ with center $\BFx$.

For a sequence of random vectors $\{\BFX_{k}\}$, 
%$\BFX_{k}\xrightarrow{D}\BFX$ denotes convergence in distribution, $\BFX_{k}\xrightarrow{p}\BFX$ denotes convergence in probability, and 
$\BFX_{k}\xrightarrow{wp1}\BFX$ denotes convergence with probability one or almost-sure convergence. For a sequence of real numbers $\{a_k\}$, we say $a_k=o\left(1\right)$ if $\lim_{k\rightarrow\infty}a_k=0$; we say $a_k=\mcO\left(1\right)$ if $\{a_k\}$ is bounded, that is, there exists a constant $M>0$ such that $\left|a_k\right|<M$ for large enough $k$. For sequences of real numbers $\{a_k\}$, $\{b_k\}$, we say that $a_k\sim b_k$ if $\lim_{k\rightarrow\infty} a_k/b_k=1$. For a sequence of random variables $\{\BFX_{k}\}$, we say $\BFX_{k} = \mcO_p(1)$ if $\{\BFX_{k}\}$ is stochastically bounded, that is, given $\epsilon > 0$ there exists $M(\epsilon) \in \real$ such that $\mbP\{\BFX_{k} \in (-M(\epsilon), M(\epsilon))\} \geq 1 - \epsilon$ for all $k \geq K(\epsilon) \in \mbN$. For a sequence of sets $\{A_n\}$ defined on a probability space, the set $\mbP\left\{A_n\, \mathrm{ i.o.}\right\} \defn \mbP\left\{\bigcap_{n=1}^{\infty} \bigcup_{m=n}^{\infty} A_m\right\}$ refers to the event that ``$A_n$ happens infinitely often."

\subsection{Key Definitions}
The following definitions will be invoked heavily during our exposition and analysis of ASTRO-DF. For further details on these definitions, consult~\cite{conschvic2009} and~\cite{ortrhe1970}.
\begin{definition}(Poised and $\Lambda$-Poised Sets)
\label{def:poised}
Given $\BFx\in \real^d$ and $\Delta>0$, let $\mcY=\left\{\BFY_i\in\mcB\left(\BFx;\Delta\right),\ i=1,2,\ldots,p\right\}$ be a finite set and $\Phi\left(\BFz\right) = \left(\phi^1\left(\BFz\right), \phi^2\left(\BFz\right), \ldots, \phi^q\left(\BFz\right)\right)$ be a polynomial basis on $\real^d$. Define 
\begin{equation}\label{int_matrix} P\left(\Phi,\mcY\right)=
\left[\begin{array}{cccc}
\phi^{1}\left(\BFY_{1}\right) & \phi^{2}\left(\BFY_{1}\right) & \ldots & \phi^{q}\left(\BFY_{1}\right)\\
\phi^{1}\left(\BFY_{2}\right) & \phi^{2}\left(\BFY_{2}\right) & \ldots & \phi^{q}\left(\BFY_{2}\right)\\
\vdots & \vdots & \vdots & \vdots\\
\phi^{1}\left(\BFY_{p}\right) & \phi^{2}\left(\BFY_{p}\right) & \ldots & \phi^{q}\left(\BFY_{p}\right)\end{array}\right].\end{equation} Then, $\mcY$ is said to be a ``poised set" in $\mcB\left(\BFx;\Delta\right)$ if the matrix $P\left(\Phi,\mcY\right)$ is nonsingular. A poised set $\mcY$ is said to be ``$\Lambda$-poised" in $\mcB\left(\BFx;\Delta\right)$ if \[\Lambda\geq \max_{j=1,\ldots,p}\max_{\BFz\in\mcB\left(\BFx;\Delta\right)}\left|\ell_{j}\left(\BFz\right)\right|,\] where $\ell_{j}\left(\BFz\right)$ are the Lagrange polynomials associated with $\mcY$. 
\end{definition}

%\begin{remark}\label{rem:lagrange} Given the matrix $P\left(\Phi,\mcY\right)$, %the value of the Lagrange polynomials can be calculated by
%\[\left|\ell_{i}\left(\BFz\right)\right|=\frac{\left|\mathrm{det}\left(P\left(\Phi,%\mcY_{i}\left(\BFz\right)\right)\right)\right|}{\left|\mathrm{det}\left(P%\left(\Phi,\mcY\right)\right)\right|}=\frac{\mathrm{vol}\left(\Phi\left(\mathcal%{Y}_{i}\left(\BFz\right)\right)\right)}{\mathrm{vol}\left(\Phi\left(\mcY%\right)\right)},\] where $\mcY_{i}\left(\BFz\right)=\mcY\backslash%\left\{ \BFY_{i}\right\} \cup\BFz$. In other words the absolute value of %the $i^{th}$ Lagrange polynomial at a given point $\BFz$ is the change in the %volume of the simplex of vertices in $\Phi\left(\mcY\right)$ when $\BFz$ %replaces $\BFY_i$ ~\cite[p. 41]{conschvic2009}.
%\end{remark}

\begin{definition}(Polynomial Interpolation Models)
\label{def:interpolation}
Let $f:\real^d \subseteq\real^d\rightarrow\real$ be a real-valued function and let $\mcY$ and $\Phi$ be as defined in Definition \ref{def:poised} with $p=q$. Suppose we can find $\BFalpha=\left(\alpha^1,\alpha^2,\ldots,\alpha^p\right)$ such that \begin{equation} \label{alpha} P\left(\Phi,\mcY\right)\BFalpha=\left(f\left(\BFY_{1}\right),\ldots,f\left(\BFY_{p}\right)\right)^{T}.\end{equation} (Such an $\BFalpha$ is guaranteed to exist if $\mcY$ is poised.) Then the function $m(\BFz):\mcB\left(\BFx;\Delta\right)\rightarrow\real$ given by \begin{equation}\label{interpolate} m(\BFz)=\sum_{j=1}^{p}\alpha^j \phi^{j}\left(\BFz\right)\end{equation} is said to be a polynomial interpolation model of $f$ on $\mcB\left(\BFx;\Delta\right)$. As a special case, $m\left(\BFz\right)$ is said to be a linear interpolation model of $f$ on $\mcB\left(\BFx;\Delta\right)$ if $\Phi(\BFz) \defn (\phi^1,\phi^2,\ldots,\phi^p) = (1,z^1,z^2,\ldots,z^d)$, and a quadratic interpolation model of $f$ on $\mcB\left(\BFx;\Delta\right)$ if $\Phi(\BFz) \defn (\phi^1,\phi^2,\ldots,\phi^p) = \left(1,z^1,z^2,\ldots,z^d,\frac{1}{2}(z^1)^2,z^1 z^2,\ldots,\frac{1}{2}(z^2)^2,\ldots,\frac{1}{2}(z_d)^2\right)$.
\end{definition}

\begin{definition}(Stochastic Interpolation Models)
\label{def:stoch-linear-interpolation}
%Recall that $\mcY=\left\{\BFY_i\in\mcB\left(\BFx;\Delta\right),\ i=1,2,\ldots,p\right\}$. 
A model constructed as in Definition \ref{def:interpolation} but with sampled function estimates is called a stochastic interpolation model. Specifically, analogous to (\ref{alpha}), suppose $\hat{\BFalpha}=\left(\alphahat^1,\alphahat^2,\ldots,\alphahat^p\right)$ is such that $$P\left(\Phi,\mcY\right)\hat{\BFalpha}=\left(\Fbar\left(\BFY_{1},n\left(\BFY_{1}\right)\right),\Fbar\left(\BFY_{2},n\left(\BFY_{2}\right)\right),\ldots,\Fbar\left(\BFY_{p},n\left(\BFY_{p}\right)\right)\right)^{T},$$ Then the stochastic function $M(\BFz):\mcB\left(\BFx;\Delta\right)\rightarrow\real$ given as $M(\BFz)=\sum_{j=1}^{p}\alphahat^j \phi^{j}\left(\BFz\right)$ is said to be a stochastic polynomial interpolation model of $f$ on $\mcB\left(\BFx;\Delta\right)$, where $\Phi\left(\BFz\right) = \left(\phi^1\left(\BFz\right), \phi^2\left(\BFz\right), \ldots, \phi^q\left(\BFz\right)\right)$ and $P\left(\Phi,\mcY\right)$ are as in Definition \ref{def:interpolation}.   

\end{definition}

\begin{definition}(Fully-linear and Fully-quadratic Models)
\label{def:flfq}
Given $\BFx\in\real^d$, $m\left(\BFz\right):\mcB\left(\BFx;\Delta\right)\rightarrow\real$, $m\in\mathcal{C}^1$ is said to be a $\left(\kappa_{ef},\kappa_{eg}\right)$-fully-linear model of $f$ on $\mcB\left(\BFx;\Delta\right)$ if it has a Lipschitz continuous gradient with Lipschitz constant $\nu_{gL}^{m}$, and there exist constants $\kappa_{ef}, \kappa_{eg}$ (not dependent on $\BFz$ and $\Delta$) such that 
\begin{equation}
\begin{aligned}\left|f\left(\BFz\right)-m\left(\BFz\right)\right| & \leq\kappa_{ef}\Delta^{2};\\
\left\| \nabla f\left(\BFz\right)-\nabla m\left(\BFz\right)\right\|  & \leq\kappa_{eg}\Delta.
\end{aligned}
\label{eq:FL}
\end{equation} Similarly $m(\BFz):\mcB\left(\BFx;\Delta\right)\rightarrow\real$, $m\in\mathcal{C}^2$ is said to be a $\left(\kappa_{ef},\kappa_{eg},\kappa_{eh}\right)$-fully-quadratic model of $f$ on $\mcB\left(\BFx;\Delta\right)$ if it has a Lipschitz continuous second derivative with Lipschitz constant $\nu_{hL}^{m}$, and constants $\kappa_{ef}$, $\kappa_{eg}, \kappa_{eh}$ (not dependent on $\BFz$, $\Delta$) such that
\begin{equation}
\begin{aligned}\left|f\left(\BFz\right)-m\left(\BFz\right)\right| & \leq\kappa_{ef}\Delta^{3};\\
\left\| \nabla f\left(\BFz\right)-\nabla m\left(\BFz\right)\right\|  & \leq\kappa_{eg}\Delta^{2};\\
\left\| \nabla^2 f\left(\BFz\right)-\nabla^2 m\left(\BFz\right)\right\|  & \leq\kappa_{eH}\Delta.
\end{aligned}
\label{eq:FQ}
\end{equation}
\end{definition} A linear interpolation model constructed using a poised set $\mcY$ can be shown to be $\left(\kappa_{ef},\kappa_{eg}\right)$-fully-linear; likewise, a quadratic interpolation model constructed using a poised set $\mcY$ can be shown to be $\left(\kappa_{ef},\kappa_{eg},\kappa_{eh}\right)$-fully-quadratic.

\begin{definition}(Cauchy Reduction)
\label{def:cauchy}
Step $\BFs$ is said to achieve $\kappa_{fcd}$ fraction of Cauchy reduction for $m\left(\cdot\right)$ on $\mcB\left(\BFx;\Delta\right)$ with some $\Delta>0$, if
\begin{equation}
m\left(\BFx\right)-m\left(\BFx+\BFs\right)\geq\frac{\kappa_{fcd}}{2}\left\| \nabla m\left(\BFx\right)\right\| \min\left\{ \frac{\left\| \nabla m\left(\BFx\right)\right\| }{\left\| \nabla^2 m\left(\BFx\right)\right\| },\Delta\right\},\label{eq:cauchy}
\end{equation}
where $\nabla m\left(\BFx\right)$ and $\nabla^2 m\left(\BFx\right)$ are the model gradient and the model Hessian at point $\BFx$. We assume $\left\| \nabla m\left(\BFx\right)\right\| / \left\| \nabla^2 m\left(\BFx\right)\right\|=+\infty$ when $\nabla^2 m\left(\BFx\right)=\mathbf{0}$. A Cauchy step with $\kappa_{fcd}=1$ is obtained by minimizing the model $m(\cdot)$ along the steepest descent direction within $\mcB\left(\BFx;\Delta\right)$~\cite[p. 175]{conschvic2009}. Accordingly, the Cauchy step is especially easy to obtain when $m(\cdot)$ is linear or quadratic.
\end{definition}

\subsection{Useful Results}

We will now state some basic results that will be used at various points in the paper. The first of these (Theorem \ref{thm:chow-robbins}) is a seminal result that is routinely used in sequential sampling contexts, especially when constructing confidence intervals. The second result (Theorem \ref{thm:meanasylim}) is a variation of Lemma 2 and Theorem 1 in~\cite{ghomuk1979} which, again, originally appeared in the context of sequential confidence intervals. As we shall see, we will use Theorem \ref{thm:meanasylim}  extensively in our analysis, to analyze the behavior of estimators that are constructed using sequential sampling. The postulates and the setting of Theorem \ref{thm:meanasylim} differ only a little from the setting of Theorem 1 in~\cite{ghomuk1979}; for this reason, we have chosen not to include a proof.  

\begin{theorem}[Chow and Robbins, 1965] \label{thm:chow-robbins} 
Suppose random variables $X_i, i =1,2,\ldots$ are iid with variance $\sigma^2<\infty$, $\bar{X}_n=n^{-1}\sum_{i=1}^{n}X_i$, $\sigmahat_{n}^{2}=n^{-1}\sum_{i=1}^{n}\left(X_i-\bar{X}_n\right)^{2}$, and $\{a_n\}$ a sequence of positive constants such that $a_n \rightarrow a$ as $n \rightarrow \infty$. If \[N(d)=\inf\left\{ n \geq 1:\frac{\sigmahat_{n}}{\sqrt{n}}\leq\frac{d}{a_{n}}\right\}, \] then $d^2N(d)/\left(a^2\sigma^2\right)\xrightarrow{wp1}1$ and $\sigmahat_{N}/\sigma\xrightarrow{wp1}1$ as $d\rightarrow 0$.
\end{theorem}

\begin{theorem}\label{thm:meanasylim} 
Suppose random variables $X_i, i =1,2,\ldots$ are iid with $\mathbb{E}[X_1] = 0, \mathbb{E}[X_1^2] = \sigma^2 >0,$ and $\mathbb{E}[|X_1|^{4v}] < \infty$ for some $v \geq 2$. Let $\sigmahat_{n}^{2}=n^{-1}\sum_{i=1}^{n}\left(X_i-\bar{X}_n\right)^{2}$, where $\bar{X}_n=n^{-1}\sum_{i=1}^{n}X_i$. If \[N(\lambda)=\inf\left\{ n \geq \lambda^{\gamma}:\frac{\sigmahat_n}{\sqrt{n}}\leq\frac{\kappa}{\sqrt{\lambda}}\right\}, \gamma \in (0,1] \] then the following hold. \begin{enumerate}\item[(i)] As $\lambda \to \infty$, $$\mbP\{N(\lambda) < \infty\} = 1 \mbox{ and } N(\lambda) \xrightarrow{wp1} \infty.$$ \item[(ii)] As $\lambda \to \infty$ and for every $s < v$, $$\mathbb{E}[N^s(\lambda)] \sim \sigma^{2s}\kappa^{-2s}\lambda^s.$$ \item[(iii)] For every $\epsilon \in (0,1)$, $$\mbP\{N(\lambda) \leq \sigma^2\kappa^{-2}\lambda (1-\epsilon)\} = \beta\lambda^{-(v-1)\gamma}.$$ where, 
as~\cite{ghomuk1979} notes, $\beta$ is a generic constant that might depend only on $v$ and the moments of $X_1$ but not on $\lambda$. \item[(iv)] As $\lambda \to \infty$, $$\mathbb{E}[\bar{X}_{N(\lambda)}^2(\lambda)] \sim \kappa^{2}\lambda^{-1}.$$ \end{enumerate}
\end{theorem} 

We next state Lemma \ref{lem:model_error} which characterizes the error in the stochastic interpolation model introduced in Definition \ref{def:stoch-linear-interpolation}. Lemma \ref{lem:model_error} is essentially a stochastic variant of a result that appears in Chapter 3 of~\cite{conschvic2009}. We provide a sketch of the proof of Lemma \ref{lem:model_error} in the Appendix.

\begin{lemma} \label{lem:model_error} Let $\mcY=\left\{ \BFY_{1},\BFY_{2},\ldots,\BFY_{p}\right\} $ be a $\Lambda$-poised set on $\mcB\left( \BFY_{1};\Delta\right)$. Let $m\left(\BFz\right)$ be an interpolation model of $f$ on $\mcB\left( \BFY_{1};\Delta\right)$. Let $M\left(\BFz\right)$ be the corresponding stochastic interpolation model of $f$ on $\mcB\left(\BFY_{1};\Delta\right)$ constructed using observations $\Fbar\left(\BFY_{i},n(\BFY_{i})\right)=f\left(\BFY_{i}\right)+E_{i}$ for $i=1,2,\ldots,p$. 

\begin{enumerate} \item[(i)] For all $\BFz\in\mcB\left( \BFY_{1};\Delta\right)$, 
\[
\left|M\left(\BFz\right)-m\left(\BFz\right)\right|\leq p\Lambda \max_{i=1,2,\ldots,p}\left|\Fbar\left(\BFY_{i},n(\BFY_{i})\right)-f\left(\BFY_{i}\right)\right|.
\] \item[(ii)] If $M\left(\BFz\right)$ is a stochastic linear interpolation model of $f$ on $\mcB\left( \BFY_{1};\Delta\right)$, then there exist positive constants $\kappa_{egL1}, \kappa_{egL2}$ such that for $\BFz\in\mcB\left( \BFY_{1};\Delta\right),$
$$
\left\| \nabla M\left(\BFz\right)-\nabla f\left(\BFz\right)\right\|   \leq\kappa_{egL1}\Delta+\kappa_{egL2}\frac{\sqrt{\sum_{i=2}^{d+1}\left(E_{i}-E_{1}\right)^{2}}}{\Delta}.$$

If $M\left(\BFz\right)$ is a stochastic quadratic interpolation model of $f$ on $\mcB\left( \BFY_{1};\Delta\right)$, then there exist positive constants $\kappa_{egQ1}, \kappa_{egQ2}$ such that
$$\left\| \nabla M\left(\BFz\right)-\nabla f\left(\BFz\right)\right\| \leq\kappa_{egQ1}\Delta^{2}+\kappa_{egQ2}\frac{\sqrt{\sum_{i=2}^{(d+1)(d+2)/2}\left(E_{i}-E_{1}\right)^{2}}}{\Delta}.$$ 
\end{enumerate}

\end{lemma} 

We end this section by stating two basic results that we repeatedly invoke, and which can be found in most standard treatments of probability such as~\cite{bil95}. The first of these is used to upper bound the probability of the union of events; the second provides sufficient conditions to ensure that an infinite number of a collection $\{A_n\}$ of events happening is zero. 

%%%%% ADDED 2/2/16
\begin{lemma}[Boole's Inequality]
\label{lem:boole}
Let $A_1,A_2,\cdots$ be a countable set of events defined on a probability space. Then $\mbP\left(\bigcup_{i}A_{i}\right)\leq\sum_{i}\mbP\left(A_{i}\right).$ Particularly, we see that if the random variables $X, X_i,\ i=1,2,\cdots,q$ satisfy $X\leq X_1+X_2+\cdots+X_q$, then 
\[\begin{aligned}\ \ \ \ \ \ \left(X>c\right) & \subseteq\left(X_{1}+X_{2}+\cdots X_{q}>c\right)\\
 & \subseteq\left(X_{1}>\frac{c}{q}\right)\cup\left(X_{2}>\frac{c}{q}\right)\cup\cdots\cup\left(X_{q}>\frac{c}{q}\right),\end{aligned}\] implying (from Boole's inequality) that 
$$\mbP\left\{ X>c\right\}  \leq\mbP\left\{ \bigcup_{i=1}^{q}\left(X_{i}>\frac{c}{q}\right)\right\} \leq\sum_{i=1}^{q}\mbP\left\{ X_{i}>\frac{c}{q}\right\}.$$
\end{lemma}

\begin{lemma}[Borel-Cantelli's First Lemma] \label{lem:borel_cantelli}
For a sequence $A_1, A_2, \ldots$ of events defined on a probability space, if $\sum_{n=1}^{\infty}\mbP\left\{A_n\right\}<\infty$, then the probability $\mbP\left\{A_n \ \mathrm{i.o.}\right\}$ of $A_n$ happening ``infinitely often" is $$\mbP\left\{A_n \ \mathrm{i.o.}\right\} \defn \mbP\left\{\bigcap_{n=1}^{\infty} \bigcup_{m=n}^{\infty} A_m\right\}=0.$$
\end{lemma}

%%%%% ADDED 9/2/16

\section{RELATED WORK}
Much progress has been made in recent times on solving various flavors of the SO problem. The predominant solution methods in the simulation literature fall into two broad categories called Stochastic Approximation (SA) and Sample-Average Approximation (SAA). SA and SAA have enjoyed a long history with mature theoretical and algorithmic literature. More recently, newer classes of algorithms that can be described as ``stochastic versions" of iterative structures in the deterministic context have emerged. 

\subsection{SA and SAA}

%Two key approaches have emerged in the simulation literature to solve SO problems: Stochastic Approximation (SA) and Sample Average Approximation (SAA). 
Virtually all stochastic approximation type methods are subsumed by the following generic form: \begin{equation}\label{sa}\BFX_{k+1} = \Pi_{\mathcal{D}} \left(\BFX_{k} - a_k \BFG_{k} \right),\end{equation} where $\Pi_{\mathcal{D}}(\BFx)$ is the projection of the point $\BFx$ onto the set $\mathcal{D}$, $\{a_k\}$ is a user-chosen positive-valued scalar sequence, and $\BFG_{k}$ is an estimator of the gradient $\nabla f(\BFX_{k})$ of the function $f$ at the point $\BFX_{k}$. When direct Monte Carlo observations of the objective function $f$ are available, the most common expression for $\BFG_{k}=\left(G_{k}^{1},G_{k}^{2},\cdots,G_{k}^{d}   \right)$ is either the central-difference approximation $G_{k}^{i} = (2c_{k})^{-1} (F(\BFX_{k}+ c_{k}\BFe_{i}) - F(\BFX_{k}- c_{k}\BFe_{i}))$ or the forward-difference approximation $G_{k}^{i} = c_{k}^{-1} (F(\BFX_{k}+ c_{k}\BFe_{i}) - F(\BFX_{k}))$ of the gradient $\nabla f(\BFX_{k})$, where $\{c_{k}\}$ is a positive-valued sequence and $F: \real^d \to \real$ is the observable estimator of the objective function $f:\real^d \to \real$. The resulting recursion is the famous Kiefer-Wolfowitz process~\cite{kw1,blum1}. More recent recursions include an estimated Hessian $H_k(\cdot)$ of the function $f$ at the point $\BFX_{k}$: %as part of $\BFG_{k}$: 
\begin{equation}\label{hessian}
%EDIT 7/12/16
%G_{k}^{i} = (2c_k)^{-1}H_k^{-1}(\BFX_{k})\left(F(\BFX_{k}+ c_{k}\BFe_{i}) - F(\BFX_{k}- c_{k}\BFe_{i})\right),
\BFX_{k+1} = \Pi_{\mathcal{D}} \left(\BFX_{k} - a_k H_k^{-1}\BFG_{k} \right),
\end{equation} making the resulting recursion in (\ref{sa}) look closer to the classical Newton's iteration in the deterministic context. The Hessian estimator $H_k(\cdot)$ has $d^2$ entries, and hence, most methods that use (\ref{hessian}) %in (\ref{sa}) 
estimate $H_k(\cdot)$ either using a parsimonious design (e.g.,~\cite{spall3,spall4}), or construct it from the history of observed points.  

As can be seen in (\ref{sa}), the SA recursion is simply stated and implemented, and little has changed in its basic structure since 1951, when it was first introduced by Robbins and Monro~\cite{robbins1} for the context of finding a zero of a ``noisy" vector function. Instead, much of the research over the ensuing decades has focused on questions such as convergence and convergence rates of SA type algorithms, the effect of averaging on the consistency and convergence rates of the iterates, and efforts to choose the sequence $\{a_k\}$ in an adaptive fashion. Some good entry points into the vast SA literature include~\cite{kushner2,mokpel11,polyjudi92,pasgho13}. See~\cite{broadie2,broadie3,yousefian2012stochastic} recent attempts to address the persistent dilemma of choosing the gain sequence $\{a_k\}$ to ensure good practical performance.

SAA, in contrast to SA, is more a framework than an algorithm to solve SO problems. Instead of solving Problem $P$, SAA asks to solve a ``sample-path" Problem $P_{n}$ (to optimality) to obtain a solution estimator $\BFX_{n}$. Formally, in the unconstrained context, SAA seeks to solve
\begin{equation} \mbox{Problem } P_{n} : \mbox{ minimize } f\left(\BFx,n\right) \mbox{ subject to } \BFx\in\real^d,\end{equation} where $f\left(\BFx,n\right)$ is computed using a ``fixed" sample of size $n$.

SAA is attractive in that Problem $P_{n}$ becomes a \textit{deterministic} optimization problem and SAA can bring to bear all of the advances in deterministic nonlinear programming methods~\cite{bazaara1,nocwri2006} of the last few decades. SAA has been the subject of a tremendous amount of theoretical and empirical research over the last two decades. For example, the conditions that allow the transfer of structural properties from the sample-path to the limit function $f(\BFx)$~\cite[Propositions 1,3,4]{kimpashen14}; the sufficient conditions for the consistency of the optimal value and solution of Problem $P_{n}$ assuming the numerical procedure in use within SAA can produce global optima~\cite[Theorem 5.3]{shadenrus09}; consistency of the set of stationary points of Problem $P_{n}$~\cite{shadenrus09,bastin1}; convergence rates for the optimal value~\cite[Theorem 5.7]{shadenrus09} and optimal solution~\cite[Theorem 12]{kimpashen14}; expressions for the minimum sample size $m$ that provides probabilistic guarantees on the optimality gap of the sample-path solution~\cite[Theorem 5.18]{rusz1}; methods for estimating the accuracy of an obtained solution~\cite{makmorwoo99,baymor07,baymor09}; and quantifications of the trade-off between searching and sampling~\cite{roysze_or2011}, have all been thoroughly studied. SAA is usually not implemented in the vanilla form $P_{n}$ due to known issues relating to an appropriate choice of the sample size $n$. There have been recent advances~\cite{pasgho13,deng2009variable1,baymor07,baymor09, baypie12} aimed at defeating the issue of sample size choice.

\subsection{Stochastic TRO}

Two algorithms that are particularly noteworthy competitors to what we propose here are STORM~\cite{chemensch2015} and the recently proposed algorithm by Larson and Billups~\cite{larbil2014} (henceforth LB2014). While the underlying logic in both of these algorithms key differences arise in terms of what has been assumed about the quality of the constructed models and how such quality can be achieved in practice. Another notable difference is that STORM also treats the context of biased estimators, that is, contexts where $\mbE\left[f(\BFx,n)\right]\neq f(\BFx)$. A key postulate that guarantees consistency in STORM is that the constructed models are of a certain specified quality (characterized through the notion of probabilistic full linearity) with a probability exceeding a fixed threshold. The authors provide a way to construct such models using function estimates constructed as sample means. Crucially, the prescribed sample means in STORM use a sample size that is derived using the Chebyshev inequality with an assumed upper bound on the variance. By contrast, the sample sizes in ASTRO-DF are determined adaptively by balancing squared bias and variance estimates for the function estimator. While this makes the sample size in ASTRO-DF a stopping time~\cite{bil95} thereby complicating proofs, such adaptive sampling enables ASTRO-DF to differentially sample across the search space, leading to efficiency.

LB2014, like STORM, uses random models. Unlike STORM, however, the sequence of models constructed in LB2014 are assumed to be accurate (as measured by a certain rigorous notion) with a probability sequence that converges to one. A related version of LB2014 ~\cite{billargra2013} addresses the case of differing levels of (spatial) 
stochastic error through the use of weighted regression schemes, where the weights are chosen heuristically. 

Another noteworthy algorithm for the context we consider in this paper is VNSP, proposed by Deng and Ferris~\cite{deng2006adaptation,deng2009variable1}.  VNSP uses a quadratic interpolation model within a trust-region optimization framework, and is derivative-free in the sense that only function estimates are assumed to be available. Model construction, inference, and improvement, along with (nondecreasing) sample size updates happen within a Bayesian framework with an assumed Gaussian conjugate prior. Convergence theory for VNSP is accordingly within a Bayesian setting. 

In the slightly more tangential context where unbiased gradient estimates are assumed to be available, a number of trust-region type algorithms have emerged in the last decade or so. STRONG or Stochastic Trust-Region Response-Surface Method 
\cite{chang2013stochastic}, for instance, is an adaptive sampling trust-region algorithm for solving SO problems that is in the spirit of what we propose here. A key feature of STRONG is local model construction through a design of experiments combined with a hypothesis testing procedure. STRONG assumes that the error in the derivative observations are additive and have a Gaussian distribution. Amos et. al. \cite{amos2014algorithm} and Bastin et. al. \cite{bastin2006adaptive} are two other examples of algorithms that treat the setting where unbiased observations of the gradient are assumed to be available. (The former, in fact, assumes that unbiased estimates of the Hessian of the objective function are available.)  Bastin et. al. \cite{bastin2006adaptive} is specific to the problem of estimation within mixed-logit models. 

\section{ASTRO--DF OVERVIEW AND ALGORITHM LISTING}
 \label{sec:overview}

ASTRO-DF is an adaptive sampling trust-region derivative-free algorithm whose essence is encapsulated within four repeating stages: (i) local stochastic model construction and certification through adaptive sampling; (ii) constrained optimization of the constructed model for identifying the next candidate solution; (iii) re-estimation of the objective function at the next candidate solution through adaptive sampling; and (iv) iterate and trust-region update based on a (stochastic) sufficient decrease check. These stages appear with italic labels in Algorithm \ref{alg:mainlisting}. In what follows, we describe each step of Algorithm \ref{alg:mainlisting} in further detail.

\begin{algorithm}[tb]  \myalgsize
\caption{ASTRO-DF Main Algorithm}
\label{alg:mainlisting}
\begin{algorithmic}[1]
\REQUIRE Initial guess $\BFx_{0}\in\real^d$, initial trust-region radius $\Deltatilde_{0}>0$ and maximum radius $\Delta_{\max}>0$, model ``fitness'' threshold $\eta_1 > 0$, trust-region expansion constant $\gamma_1 > 1$ and  contraction constant $\gamma_2 \in (0,1)$, initial sample size $n_0$, sample size lower bound sequence $\{\lambda_k\}$ such that $k^{(1+\epsilon)} = \mcO(\lambda_k)$,  initial sample set $\mcYtilde_0=\left\{\BFx_{0}\right\}$, and outer adaptive sampling constant $\kappa_{oas}$. 
\FOR{$k=0,1,2,\hdots$}
\STATE \label{ASTRO:construct} \textit{Model Construction:} Construct the model at $\BFX_k$ by calling Algorithm~\ref{alg:sublisting} with the candidate 
\STATEx \hspace{0.1in} trust-region radius $\Deltatilde_k$ and 
candidate set of sample points $\mcYtilde_k$,  \[[M_k(\BFX_k+\BFs),\Delta_k,\mcY_k]={\tt{AdaptiveModelConstruction}}(\Deltatilde_k,\mcYtilde_k).\] 
\STATEx \hspace{0.1in} Set $\Ntilde_{k}=N\left(\BFX_{k}\right)$.
\STATE \label{ASTRO:TRsubprob} \textit{TR Subproblem:} Approximate the $k$th step by minimizing the model in the trust-region, 
\STATEx \hspace{0.1in} $\BFS_{k}=\argmin_{\left\| \BFs\right\| \leq\Delta_{k}}M_{k}(\BFX_k+\BFs)$, and set the new candidate point $\BFXtilde_{k+1}=\BFX_k+\BFS_{k}$. 
\STATE \label{ASTRO:evaluate}\textit{Evaluate:} Estimate the function at the candidate point using adaptive sampling to obtain 
\STATEx \hspace{0.1in} $\Fbar(\BFXtilde_{k+1},\Ntilde_{k+1})$, where 
 \begin{align}
 \Ntilde_{k+1}=\max\biggl\{ \lambda_{k},\min\biggl\{ n:\frac{\sigmahat_{F}\left(\BFXtilde_{k+1},n\right)}{\sqrt{n}}\leq\frac{\kappa_{oas}\Delta_{k}^{2}}{\sqrt{\lambda_{k}}}\biggr\} \biggr\},\label{eq:outer-sampling}
\end{align} 
\STATEx\hspace{0.1in}\textit{Update:}
\STATE \label{ASTRO:ratio} Compute the success ratio $\rhohat_{k}$ as
\[
\begin{aligned}\rhohat_{k}=\frac{\Fbar\left(\BFX_{k},\Ntilde_{k}\right)-\Fbar\left(\BFXtilde_{k+1},\Ntilde_{k+1}\right)}{M_{k}(\BFX_{k})-M_{k}(\BFXtilde_{k+1})}.\end{aligned}
\]
\IF{$\rhohat_k>\eta_1$}\label{ASTRO:accept} 
\STATE $\BFX_{k+1}=\BFXtilde_{k+1},$ $\Deltatilde_{k+1}=\min\{\gamma_1\Delta_k,\Delta_{\max}\}$, $N_{k+1}=\Ntilde_{k+1}$. 
\STATEx \hspace{.22in} Update the sample set $\mcYtilde_{k+1}$ to include the new iterate.
%\STATEx \hspace{.22in} Set $\BFYtilde_{\max}:=\arg\max_{\BFY_{i}\in\mathcal{Y}_{k}}\left\{ \left\| \BFXtilde_{k+1}-\BFY_{i}\right\| \right\}. $ Update the sample set 
%\[\mcYtilde_{k+1}=\mathcal{Y}_{k} \backslash \left\{\BFYtilde_{\max}\right\} \cup \left\{\BFX_{k+1}\right\}.\]
\ELSE 
\STATE \label{ASTRO:reject} $\BFX_{k+1}=\BFX_k,$ $\Deltatilde_{k+1}=\gamma_2\Delta_k$, $N_{k+1}=\Ntilde_{k}$. 
\STATEx \hspace{.22in} Update the sample set $\mcYtilde_{k+1}$, if needed, to include the rejected candidate point.
%\STATEx \hspace{.22in} Set $\BFY_{\max}:=\arg\max_{\BFY_{i}\in\mathcal{Y}_{k}}\left\{ \left\| \BFX_{k}-\BFY_{i}\right\| \right\} $. If $\BFXtilde_{k+1}\neq\BFY_{\max}$, update 
% \[\mcYtilde_{k+1}=\mathcal{Y}_{k} \backslash \left\{\BFY_{\max}\right\} \cup \left\{\BFXtilde_{k+1}\right\}.\]
\ENDIF 
\ENDFOR
\end{algorithmic}
\end{algorithm}

\begin{algorithm}[htb]  \myalgsize
\caption{$[M_k(\BFX_k+\BFs),\Delta_k,\mcY_k]$={\tt{AdaptiveModelConstruction}}($\Deltatilde_k,\mcYtilde_k$)}
\label{alg:sublisting}
\begin{algorithmic}[1]
\REQUIRE \emph{Parameters from ASTRO-DF:} candidate trust-region radius $\Deltatilde_k$ and candidate sample set $\mcYtilde_k$ (possibly with cardinality $<p$).
\STATEx \emph{Parameters specific to \tt{AdaptiveModelConstruction}:} trust-region contraction factor $w \in (0,1)$, trust-region and gradient balance constant $\mu$, gradient inflation constant $\beta$ with $0<\beta<\mu$, and inner adaptive sampling constant $\kappa_{ias}$.
 \STATE Initialize $j_k=1$, set $\mcY_k^{(j_{k})}=\mcYtilde_k$, and set $\BFY_{1}=\BFX_{k}$ where $\BFX_{k}$ is the first element of $\mcYtilde_k$.
\STATEx\hspace{0in}\textit{Contraction loop:}
 \REPEAT
\STATE  \label{AMC:sample_set} Improve $\mcY_{k}^{(j_{k})}=\left\{ \BFY_{1}^{(j_{k})},\BFY_{2}^{(j_{k})},\ldots,\BFY_{p}^{(j_{k})}\right\}$ by appropriately choosing $\BFY_{i}^{(j_{k})},\ i=2,3,\cdots,p$ 
\STATEx \hspace{0.1in} to make it a poised set in $\mcB(\BFX_{k};\Deltatilde_{k}w^{j_k-1})$. 
\FOR{$i=1$ to $p$} \label{AMC:sampling_begin}
\STATE \label{AMC:estimate}\label{former1b}Estimate $\Fbar\left(\BFY_{i}^{(j_{k})},N\left(\BFY_{i}^{(j_{k})}\right)\right)$, where
\begin{equation}
N\left(\BFY_{i}^{(j_{k})}\right)=\max\biggl\{\lambda_k,\min\biggl\{ n:\frac{\sigmahat_{F}\left(\BFY_{i},n\right)}{\sqrt{n}}\leq\frac{\kappa_{ias}(\Deltatilde_{k}w^{j_k-1})^{2}}{\sqrt{\lambda_{k}}}\biggr\}\biggr\}.\label{eq:inner-sampling-interpolation}
\end{equation}
\ENDFOR \label{AMC:sampling_end}
\STATE Construct a quadratic model $M_{k}^{\left(j_k\right)}\left(\BFX_{k}+\BFs\right)$ via interpolation.
\STATE  Set $j_k=j_k+1$.
 \UNTIL{$\Deltatilde_{k}w^{j_k-1}\leq\mu\| \nabla M_{k}^{\left(j_k\right)}\left(\BFX_{k}\right)\|$}.\label{AMC:end_loop}
\STATE Set  $M_{k}(\BFX_k+\BFs)=M_{k}^{(j_k)}(\BFX_k+\BFs)$, $\nabla M_{k}(\BFX_{k})=\nabla M_{k}^{(j_k)}(\BFX_{k})$, and $\nabla^2 M_{k}(\BFX_{k})=\nabla^2 M_{k}^{(j_k)}(\BFX_{k})$.
\STATE\label{AMC:update}\textbf{return} $M_{k}(\BFX_k+\BFs)$, $\Delta_{k}=\min\left\{ \Deltatilde_{k},\max\left\{ \beta\left\| \nabla M_{k}(\BFX_k)\right\| ,\Deltatilde_{k}w^{j_k-1}\right\} \right\} $, and $\mcY_k=\mcY_k^{(j_{k})}$.
\end{algorithmic}
\end{algorithm}

%Step 0 of Algorithm \ref{alg:mainlisting} is the initialization step where the iteration number $k$ and trust-region radius $\Delta_k$ are initialized. 

In Step \ref{ASTRO:construct} of Algorithm \ref{alg:mainlisting}, a stochastic model of the function $f(\cdot)$ in the trust-region $\mcB(\BFX_{k};\Delta_{k})$ %having specific properties 
is constructed using Algorithm \ref{alg:sublisting}. The aim of Algorithm \ref{alg:sublisting} is to construct a model of a specified quality within a trust-region having radius smaller than a fixed multiple of the model gradient norm. During the $j_{k}$th iteration of Algorithm \ref{alg:sublisting}, a poised set $\mcY_{k}^{(j_{k})} \triangleq \{\BFY_{1}^{(j_{k})}, \BFY_2^{(j_{k})}, \ldots, \BFY_p^{(j_{k})}\}$ in the ``candidate" trust-region having radius $\Deltatilde_k w^{j_{k}-1}$ and center $\BFY_{1}^{(j_{k})}=\BFX_{k} $ is chosen (Step \ref{AMC:sample_set}); Monte Carlo function estimates are then obtained at each of the points in $\mcY_{k}^{(j_{k})}$ with $N\left(\BFY_{i}^{(j_{k})}\right)$ being the sample size at point $\BFY_{i}^{(j_{k})}$ after the $j_{k}$th iteration of the contraction loop. Sampling at each point in $\mcY_{k}^{(j_{k})}$ is adaptive and continues (Steps \ref{AMC:sampling_begin}--\ref{AMC:sampling_end}) until the estimated standard errors $\sigmahat_F\left(\BFY_{i}^{(j_{k})},N\left(\BFY_{i}^{(j_{k})}\right)\right)/\sqrt{N\left(\BFY_{i}^{(j_{k})}\right)}$ of the function estimates $\Fbar\left(\BFY_{i}^{(j_{k})},N\left(\BFY_{i}^{(j_{k})}\right)\right)$ drop below a slightly inflated square of the candidate trust-region radius. A linear (or quadratic) interpolation model is then constructed using the obtained function estimates in Step \ref{AMC:estimate}. (If a linear interpolation model is constructed, $p = d+1$, and if a quadratic interpolation model is constructed, $p = (d+1)(d+2)/2$.) If the resulting model $M_k^{(j_{k})}(\BFz), \BFz \in\mcB(\BFX_{k};\Deltatilde_k w^{j_{k}-1})$ is such that the candidate trust-region radius $\Deltatilde_k w^{j_{k}-1}$ is too large compared to the norm of the model gradient $\left\| \nabla M_k^{(j_{k})}(\BFX_{k})\right\|$, that is, if $\Deltatilde_{k} w^{j_{k}-1} > \mu \left\| \nabla M_k^{(j_{k})}(\BFX_{k})\right\|$, then the candidate trust-region radius is shrunk by a factor $w$ and control is returned back to Step \ref{AMC:sample_set}. On the other hand, if the candidate trust-region radius is smaller than the product of $\mu$ and the norm of the model gradient, then the resulting stochastic model is accepted but over an updated incumbent trust-region radius given by Step \ref{AMC:update}. (Step \ref{AMC:update} of Algorithm \ref{alg:sublisting}, akin to~\cite{conn2009global}, updates the incumbent trust-region radius to the point in the interval $[\Deltatilde_k w^{j_{k}-1}, \Deltatilde_k]$ that is closest to $\beta \left\|\nabla M_k^{(j_{k})}(\BFX_{k})\right\|$).

We emphasize the following four issues pertaining to Step 2 in Algorithm \ref{alg:mainlisting} and the model resulting from the application of Algorithm \ref{alg:sublisting}.

\begin{enumerate} \item[(i)] Due to the nature of the chosen poised set $\mcY_{k}$, the (hypothetical) limiting model $m_k(\BFX_{k})$ constructed from true function observations on $\mcY_{k}$ will be $\left(\kappa_{ef},\kappa_{eg}\right)$-fully-linear (or $\left(\kappa_{ef},\kappa_{eg},\kappa_{eh}\right)$-fully-quadratic) on the updated trust-region $\mcB(\BFX_{k};\Delta_k)$. Of course, the model $m_k(\BFX_{k})$ is unavailable since true function evaluations are unavailable; and it makes no sense to talk about whether the constructed model $M_k(\BFX_{k})$ is $\left(\kappa_{ef},\kappa_{eg}\right)$-fully-linear (or $\left(\kappa_{ef},\kappa_{eg},\kappa_{eh}\right)$-fully quadratic) since it is constructed from stochastic function estimates.

\item[(ii)] By construction, the trust-region resulting from the application of Algorithm \ref{alg:sublisting} has a radius that is at most $\beta$ times the model gradient norm $\left\|\nabla M_k(\BFX_{k})\right\|$. 

\item[(iii)] The structure of adaptive sampling in Step \ref{AMC:estimate} of Algorithm \ref{alg:sublisting} is identical to that appearing for estimation in Step \ref{ASTRO:evaluate} of Algorithm \ref{alg:mainlisting}. The adaptive sampling step simply involves sampling until the estimated standard error of the function estimate comes within a factor of the deflated square of the incumbent trust-region radius. As our convergence proofs will reveal, balancing the estimated standard error to any lower power of the incumbent trust-region radius will threaten consistency of ASTRO-DF's iterates. 

\item[(iv)] As can be seen from the algorithm listings, the quality of the constructed model in ASTRO-DF is ensured (in Step 2) at every iteration. As will be evident from our analysis, such stringency eases theoretical analysis. This is in contrast to many modern implementations of derivative-free trust-region algorithms where model quality is checked and improved (if necessary) through a ``criticality step" which is triggered only when the norm of the model gradient falls below a threshold. This selective model quality assurance makes the algorithm more numerically efficient, at the price of a more complicated trust-region management and convergence analysis (See acceptable iterations and model improving iterations in 
\cite[p. 185]{conschvic2009}). We believe that incorporating a similar criticality step in ASTRO-DF will ease computational burden during implementation. 

\end{enumerate}

%\begin{remark} \label{rem:practic} 

%\end{remark}

%If the entering trust-region radius $\Deltatilde_{k}$ to Algorithm \ref{alg:sublisting} is not small enough, the loop in Step 1 guarantees that the leaving trust-region radius $\Delta_{k}$ from Algorithm \ref{alg:sublisting} is contracted by the factor $w$ to the extent that it remains in lock step with the model gradient $\nabla M_{k}\left(\BFX_{k}\right)$, set to be at least as low as $\beta$ fraction of model gradient. On the other hand if the entering trust-region radius $\Deltatilde_{k}$ is small enough no contractions are preformed. The expected model $\mathbb{E}\left[M_{k}^{\left(j\right)}\left(\BFz\right)\right]$ is $\left(\kappa_{ef},\kappa_{eg}\right)$-fully-linear on $\mcB\left(\BFX_{k};\Deltatilde_{k}w^{j-1}\right)$ as in Definition \ref{def:flfq} since a poised interpolation set is ensured in Step 1(a). We will later prove that it implies the expected model is $\left(\kappa_{ef},\kappa_{eg}\right)$-fully-linear on the whole trust-region $\mcB\left(\BFX_{k};\Delta_{k}\right)$. Another nontrivial fact is that Algorithm \ref{alg:sublisting} terminates almost surely. The formal statement and rigorous proofs of these results appear in Section \ref{sec:conv}. Also, as we shall see in Section \ref{sec:conv}, the increased sampling resulting from the use of the inflation sequence $\lambda_k$ ensures that the spurious effects of stochastic sampling decay at a fast enough rate to ensure the almost sure convergence of the iterates $\{\BFX_{k}\}$.  

Let us now resume our discussion of Algorithm \ref{alg:mainlisting}. In Step \ref{ASTRO:construct}, Algorithm \ref{alg:mainlisting} executes {\tt{AdaptiveModelConstruction}} to obtain a model $M_{k}\left(\BFz\right), \BFz \in \mcB\left(\BFX_{k}; \Delta_{k}\right)$ whose limiting approximation is $\left(\kappa_{ef},\kappa_{eg}\right)$-fully-linear (or $\left(\kappa_{ef},\kappa_{eg},\kappa_{eh}\right)$-fully-quadratic) as observed in (i) above. Step \ref{ASTRO:TRsubprob} in Algorithm \ref{alg:mainlisting} then approximately solves the constrained optimization problem $\BFS_{k}=\arg\min_{\left\| \BFs\right\| \leq\Delta_{k}}M_{k}\left(\BFX_{k}+\BFs\right)$ to obtain a candidate point $\BFXtilde_{k+1} = \BFX_{k} + \BFS_{k}$ satisfying the $\kappa_{fcd}$-Cauchy decrease as defined in Assumption \ref{assump:cauchy}. 

In preparation for checking if the candidate solution $\BFXtilde_{k+1}$ provides sufficient decrease, Step \ref{ASTRO:evaluate} of Algorithm \ref{alg:mainlisting} obtains Monte Carlo samples of the objective function at $\BFXtilde_{k+1}$, until the estimated standard error $\sigmahat_F\left(\BFXtilde_{k+1},\Ntilde_{k+1}\right)/\sqrt{\Ntilde_{k+1}}$ of $\Fbar\left(\BFXtilde_{k+1},\Ntilde_{k+1}\right)$ is smaller than a slightly deflated square of the trust-region radius $\lambda_{k}^{-1/2}\kappa_{oas}\Delta_{k}^{2}$, subject to the sample size being at least as big as $\lambda_{k}$ (see Remark \ref{rem:lambda}). 

In Step \ref{ASTRO:ratio} of Algorithm \ref{alg:mainlisting}, the obtained function estimate is used to check if the so called  \emph{success ratio} $\rhohat_{k}$, that is, the ratio of the predicted to the observed function decrease at the point $\BFXtilde_{k+1}$, exceeds a fixed threshold $\eta_1$. The denominator of the success ratio $\rhohat_{k}$ is calculated by evaluating the constructed model at $\BFXtilde_{k+1}$ using the analytical form listed in Definition \ref{def:stoch-linear-interpolation}. If $\rhohat_{k}$ exceeds the threshold $\eta_{1}$, the candidate $\BFXtilde_{k+1}$ is accepted as the new iterate $\BFX_{k+1}$, the iteration is deemed successful, and the trust-region is expanded (Step \ref{ASTRO:accept}). If $\rhohat_{k}$ falls below the specified threshold $\eta_1$, the candidate $\BFXtilde_{k+1}$ is rejected (though it may remain in the sample set), the iteration is deemed unsuccessful, and the trust-region is shrunk (Step \ref{ASTRO:reject}). In either case, $\Deltatilde_{k+1}$ is set as the incumbent trust-region radius, $N_{k+1}$ is set as the current sample size of $\BFX_{k+1}$, and $\mcY_{k}$ is set as the interpolation set for the next iteration. Note that in the next iteration the sample size of $\BFX_{k+1}$ is subject to change through Step \ref{ASTRO:construct} again.%, and the algorithm control is returned to Step \ref{ASTRO:cons}.

\begin{remark} \label{rem:lambda} The sequence $\{\lambda_k\}$ appearing as the first argument of the ``max" function in the expression for the adaptive sample size (in Step \ref{ASTRO:evaluate} of Algorithm \ref{alg:mainlisting} and Step \ref{AMC:estimate} of Algorithm \ref{alg:sublisting}) is standard for all adaptive sampling contexts, e.g.,~\cite{chorob1965,ghomuksen97}, and intended to nullify the effects of mischance without explicitly participating in the limit. Since $N=\max\left\{\lambda_k,\lambda_k\sigmahat_F^2(\BFX_k,N)\kappa_{oas}^{-2}\Delta_k^{-4}\right\}$, and since we will demonstrate that $\sigmahat_F^2(\BFX_k,N)\xrightarrow{wp1}\sigma > 0$ and $\Delta_{k}\xrightarrow{wp1}0$, the probability of the first argument in the expression for the adaptive sample size being binding will decay to zero as $k \to \infty$.
%In fact, the first argument $\lambda_k$ in the expression for the adaptive sample size can be reduced to $\lambda_k^{\gamma}$ for any $\gamma \in (0,1]$ with little effect in the ensuing proofs. We have chosen not to do this in view of not making the user choose yet another parameter.
\end{remark}

\section{CONVERGENCE ANALYSIS OF ASTRO-DF}\label{sec:conv}

As noted in Section \ref{sec:overview}, the convergence behavior of ASTRO-DF depends crucially on the behavior of three error terms expressed through the following decomposition: 
\begin{align}\label{eq:err_decomp}\left\vert \Fbar\left(\BFXtilde_{k+1},\Ntilde_{k+1}\right)-M_k\left(\BFXtilde_{k+1}\right)\right\vert \leq &\left\vert \Fbar\left(\BFXtilde_{k+1},\Ntilde_{k+1}\right)-f\left(\BFXtilde_{k+1}\right)\right\vert \nonumber \\
&+\left\vert f\left(\BFXtilde_{k+1}\right)-m_k\left(\BFXtilde_{k+1}\right)\right\vert \nonumber \\
&+\left\vert m_k\left(\BFXtilde_{k+1}\right)-M_k\left(\BFXtilde_{k+1}\right)\right\vert.
\end{align}
The three terms appearing on the right-hand side of (\ref{eq:err_decomp}) can be interepreted, respectively, as follows: (i) the \emph{stochastic sampling error} $\left\vert \Fbar\left(\BFXtilde_{k+1},\Ntilde_{k+1}\right)-f\left(\BFXtilde_{k+1}\right)\right\vert$ arising due to the fact that function evaluations are estimated using Monte Carlo; (ii) the \emph{deterministic model error} $\left\vert f\left(\BFXtilde_{k+1}\right)-m_k\left(\BFXtilde_{k+1}\right)\right\vert$ arising due to the choice of local model; and (iii) the \emph{stochastic interpolation error} $\left\vert m_k\left(\BFXtilde_{k+1}\right)-M_k\left(\BFXtilde_{k+1}\right)\right\vert$ arising due to the fact that model prediction at unobserved points is a combination of the model bias and the error in (i). (The analysis in the deterministic context involves only the error in (ii).) Accordingly, driving the errors in (i) and (ii) to zero sufficiently fast, while ensuring the fully-linear or quadratic sufficiency of the expected model, guarantees almost sure convergence. 

Driving the errors in (i) and (ii) to zero sufficiently fast is accomplished by forcing the sample sizes to increase across iterations at a sufficiently fast rate, something that we ensure by keeping the estimated standard error of all function estimates in lock step with the square of the trust-region radius. The trust-region radius is in turn also kept in lock-step with the model gradient through the model construction Algorithm \ref{alg:sublisting}. Such a deliberate lock-step between the model error, trust-region radius, and the model gradient is aimed at efficiency without sacrificing consistency. 

In what follows, we provide a formal proof of the wp1 convergence of ASTRO-DF's iterates. Recall that we assume that the models being constructed within Step \ref{ASTRO:construct} of Algorithm \ref{alg:mainlisting} are either linear or quadratic. Furthermore, we focus only on convergence to a first-order critical point of the function $f$. 

We first list six standing assumptions that are assumed to hold for the ensuing results. Additional assumptions will be made as and when required.

%\textbf{Assumption 1.} The simulation oracle calls have embedded $\mathbb{E}\left[F\left(\BFx\right)\icright]=f\left(\BFx\right)$,
%$\mathrm{Var}\left(F\left(\BFx\right)\right)=\sigma^{2}<\infty$.   

%\bigskip{}
\begin{assumption}
\label{assump:diff-bddbelow}
The function $f$ is continuously differentiable and bounded from below. Furthermore, it has Lipschitz continuous gradients, that is, there exists $\nu_{gL}$ such that $\left\| \nabla{f}\left(\BFx\right)-\nabla{f}\left(\BFy\right)\right\| \leq \nu_{gL} \left\| \BFx-\BFy\right\|$ for all $\BFx,\BFy\in\real^d$. 
\end{assumption}

\begin{assumption}\label{assump:crit}
There exists a first-order critical point associated with Problem $P$, that is, there exists $\BFx^{\ast}\in\real^d$ such that: $\nabla{f}\left(\BFx^{\ast}\right)=0$.
\end{assumption}

%\begin{assumption}
%The constraint set is vacuous, that is, in Problem $P$, $c=0$.
%\end{assumption}

\begin{assumption}
\label{assump:var}
The Monte Carlo oracle, when executed at $\BFX_{k}$, generates independent and identically distributed random variates $F_{j}(\BFX_{k}) = f(\BFX_{k}) + \xi_j \, \vert \, \mcF_k$, where $\xi_1,\xi_2, \ldots $ is a martingale-difference sequence adopted to $\mcF_k$ such that $\mbE\left[\xi_j^2 \, \vert \, \mcF_k\right]=\sigma^2$ for all $k$, where $\sigma^2<\infty$ and $\sup_k \mathbb{E}[|\xi_j|^{4v} \, \vert \, \mcF_k] < \infty$ for some $v \geq 2$. 
\end{assumption}

Assumptions \ref{assump:diff-bddbelow} -- \ref{assump:var} are arguably mild assumptions relating to the problem. The following three assumptions relate to the nature of the proposed algorithm ASTRO-DF.

\begin{assumption}
\label{assump:cauchy}
The minimizer obtained in the trust-region subproblem (Step \ref{ASTRO:TRsubprob} of Algorithm \ref{alg:mainlisting}) satisfies a $\kappa_{fcd}$-Cauchy decrease with $\kappa_{fcd} > 0$, that is,
\[
M_{k}\left(\BFX_{k}\right)-M_{k}\left(\BFXtilde_{k+1}\right)\geq\frac{\kappa_{fcd}}{2}\left\| \nabla M_{k}\left(\BFX_{k}\right)\right\| \min\left\{ \frac{\left\| \nabla M_{k}\left(\BFX_{k}\right)\right\| }{\left\| \nabla^2 M_{k}\left(\BFX_{k}\right)\right\| },\Delta_{k}\right\}.
\]\end{assumption}

\begin{assumption}
\label{assump:hessian-bound}
There exists a positive constant $\kappa_{bhm}$, such that %for every iteration $k$ and every $\BFx$, 
the model Hessian $\nabla^2 M_{k}\left(\BFX_k\right)$ satisfies 
\[
\mbP\left\{ \omega: \ \limsup_{k\rightarrow\infty}\left\| \nabla^2 M_{k}\left(\BFX_k(\omega)\right)\right\| \leq\kappa_{bhm}\right\} =1,
\] that is, the model Hessians are eventually bounded for all $\BFx$ and $k$ almost surely.
\end{assumption}

\begin{assumption} \label{assump:lambda_lbd} The ``lower-bound sequence" $\{\lambda_k\}$ is chosen to satisfy
$k^{(1+\epsilon)} = \mcO(\lambda_k)$ for some $\epsilon > 0$.
\end{assumption}

%Assumption \ref{assump:cauchy} holds automatically for linear interpolation models since the model Hesssian $\nabla^2 M_{k}\left(\BFx\right) = 0$. For quadratic interpolation models, it can be shown that if $\BFs_k$ is chosen as $\BFs_k = t_C \nabla M_k(\BFX_{k})$, where $$t_C = \argmin_{\alpha \in [0,\Delta_k]} \,\, M_k(\BFX_{k} - \alpha \nabla M_k(\BFX_{k})),$$ then $\BFs_k$ satisfies a $\frac{1}{2}$-Cauchy decrease. Steps $\BFs_k$ based on an inexact line search can also be accommodated to satisfy a $\kappa_{fcd}$ Cauchy decrease. (See Section 10.1 in~\cite{conschvic2009} for additional details.)

%%%%%% EDIT 2/1/16
Assumption \ref{assump:cauchy} is usually easily ensured through an appropriate choice in the trust-region subproblem step (Step \ref{ASTRO:TRsubprob} of Algorithm \ref{alg:mainlisting}). For example, Assumption \ref{assump:cauchy} is satisfied if the optimization resulting from Step \ref{ASTRO:TRsubprob} of Algorithm \ref{alg:mainlisting} yields a solution that is at least as good as the Cauchy step $t_C$, which is the minimizer of the model $M_k(\cdot)$ along the direction $\nabla M_k(\BFX_k)$ and constrained to the trust-region, that is, $$t_C = \argmin_{\alpha \in [0,\Delta_k]} \,\, M_k(\BFX_{k} - \alpha \nabla M_k(\BFX_{k})).$$ (See Section 10.1 in~\cite{conschvic2009} for additional details.) Likewise, Assumption \ref{assump:hessian-bound} can be enforced through a check that is performed each time the model is constructed or updated. And, Assumption \ref{assump:lambda_lbd} imposes a (weak) minimum increase on the sample sizes for estimation and model construction operations within ASTRO-DF. 

\begin{remark} It is our view that the minimum rate of increase on the lower bound sequence $\{\lambda_k\}$ can be reduced to a logarithmic increase instead of what has been assumed in Assumption \ref{assump:lambda_lbd}. Using the notation of Theorem \ref{thm:meanasylim}, this will require a large-deviation type bound on the tail probability $\mbP\{|\bar{X}_N| > t\}$ after assuming the existence of the moment-generating function of $X_i$'s. To the best of our knowledge there currently exist no such results for fixed-width confidence interval stopping, which is the context of Theorem \ref{thm:meanasylim}. \end{remark}

\subsection{Main Results}\label{sec:mainres}

We are now ready to establish the \emph{consistency}, that is, the almost sure convergence to a stationary point, of the iterates generated by ASTRO-DF. The roadmap for consistency consists of three main theorems supported by three lemmas. The first of the main theorems is Theorem \ref{thm:delta-conv} where we demonstrate that the sequence of trust-region radii $\{\Delta_k\}$ across the iterations of ASTRO-DF converges to zero with probability one. Next, Theorem \ref{thm:finite-unsuccess} establishes that the iterations within ASTRO-DF are eventually successful with probability one. The final result that establishes the almost sure convergence of ASTRO-DF's iterates to a stationary point appears as Theorem \ref{thm:conv}. With the exception of Lemma \ref{lem:bounds}, all results that follow analyze the ensemble probabilistic behavior of the sample paths to then make inferences about individual sample-paths, primarily through the first Borel-Cantelli's first lemma (see Lemma \ref{lem:borel_cantelli}). Lemma \ref{lem:bounds} is an exception in that the analysis there is pathwise, without probabilistic arguments.    

We start with Lemma \ref{lem:fhat-bounded} which establishes that the sequence of function estimates at the iterates generated by ASTRO-DF has to remain bounded with probability one. 

\begin{lemma}
\label{lem:fhat-bounded} 
Let Assumptions \ref{assump:diff-bddbelow}, \ref{assump:var}, and \ref{assump:lambda_lbd} hold. Then 
\begin{itemize}
\item[(a)] $\mbP\left\{\lim_{k\rightarrow\infty} \Fbar\left(\BFX_{k},N_{k}\right)=-\infty\right\} =0$,
\item[(b)] $\mbP\left\{ \left|\Fbar\left(\BFX_{k},N_{k}\right)-f\left(\BFX_{k}\right)\right|\geq c\Delta_k^2\:\:\mathrm{i.o.}\right\}=0$.
\end{itemize}
\end{lemma}

\begin{proof}
For part (a) we know from Assumption \ref{assump:diff-bddbelow} that $f$ is bounded from below. Hence we can write
for any $c\in\mathbb{R}$, %for any $c>\sqrt{2}\sigma$, %, there exists a large number $C$, such that $f\left(\BFX_{k}\right)\geq C$ for all $\BFx\in\mbX$. Then
\begin{equation}\label{contra}
\mbP\left\{ \lim_{k\rightarrow\infty}\Fbar\left(\BFX_{k},N_{k}\right)=-\infty\right\} 
\leq \mbP\left\{ \left|\Fbar\left(\BFX_{k},N_{k}\right)-f\left(\BFX_{k}\right)\right|\geq c\:\:\mathrm{i.o.}\right\} ,\end{equation} 
where $\rm{i.o.}$ stands for ``infinitely often." (See Section \ref{sec:notcon} for a formal definition and use of ``infinitely often.") However by the law of total probability,
\begin{align}\label{ubdio_bd}
\mbP\left\{ \left|\Fbar\left(\BFX_{k},N_{k}\right)-f\left(\BFX_{k}\right)\right|\geq c\right\}  
& = \mathbb{E}\left [\mbP\left\{ \left|\Fbar\left(\BFX_{k},N_{k}\right)-f\left(\BFX_{k}\right)\right|\geq c \, \vert \, \mcF_k \right\} \right ] \nonumber \\
& \leq \mathbb{E}[c^{-2}\mathbb{E}[ \left(\Fbar\left(\BFX_{k},N_{k}\right)-f\left(\BFX_{k}\right)\right)^2 \, \vert \, \mcF_k]],\end{align} where the last inequality follows from Chebyshev's inequality~\cite{bil95}. Now invoke part (iv) of Theorem \ref{thm:meanasylim} along with Assumption \ref{assump:var} and the sample size expression in (\ref{eq:outer-sampling}) to notice that \begin{equation}\label{mukghobd}\mathbb{E}[ \left(\Fbar\left(\BFX_{k},N_{k}\right)-f\left(\BFX_{k}\right)\right)^2 \, \vert \, \mcF_k] \sim \kappa_{oas}^2\Delta_k^4\lambda_k^{-1}\end{equation} as $\lambda_k \to \infty;$ that is, for large-enough $k$, we can write for any $\delta>0$ that \begin{align}\label{martbd}\mathbb{E}[ \left(\Fbar\left(\BFX_{k},N_{k}\right)-f\left(\BFX_{k}\right)\right)^2 \, \vert \, \mcF_k] & \leq (1+\delta)\kappa_{oas}^2\Delta_k^4\lambda_k^{-1} \nonumber \\
& \leq (1+\delta)\kappa_{oas}^2\Delta_{\max}^4\lambda_k^{-1}.\end{align} Now use (\ref{ubdio_bd}) and (\ref{martbd}) to write \begin{align} \label{eq:contra_lb} \mbP\left\{ \left|\Fbar\left(\BFX_{k},N_{k}\right)-f\left(\BFX_{k}\right)\right|\geq c\right\} & \leq \mathbb{E}[c^{-2}\mathbb{E}[ \left(\Fbar\left(\BFX_{k},N_{k}\right)-f\left(\BFX_{k}\right)\right)^2 \, \vert \, \mcF_k]] \nonumber \\ 
& \leq \mathbb{E}[c^{-2}(1+\delta)\kappa_{oas}^2\Delta_{\max}^4\lambda_k^{-1}] \nonumber \\
& = c^{-2}(1+\delta)\kappa_{oas}^2\Delta_{\max}^4\lambda_k^{-1}.\end{align} The right-hand side of (\ref{eq:contra_lb}) is summable since $k^{(1 + \epsilon)} = \mcO(\lambda_k)$ for some $\epsilon > 0$; we can thus invoke the first Borel-Cantelli Lemma~\cite{bil95} and conclude that the right-hand side of (\ref{contra}) is zero. This proves part (a).

In fact, similar to (\ref{eq:contra_lb}) we have \begin{align} \label{eq:contra_lb_Delta} \mbP\left\{ \left|\Fbar\left(\BFX_{k},N_{k}\right)-f\left(\BFX_{k}\right)\right|\geq c\Delta_k^2\right\} & \leq \mathbb{E}[c^{-2}\mathbb{E}[ \left(\Fbar\left(\BFX_{k},N_{k}\right)-f\left(\BFX_{k}\right)\right)^2 \, \vert \, \mcF_k]] \nonumber \\ 
& \leq \mathbb{E}[c^{-2}\Delta_{k}^{-4}(1+\delta)\kappa_{oas}^2\Delta_{k}^4\lambda_k^{-1}] \nonumber \\
& = c^{-2}(1+\delta)\kappa_{oas}^2\lambda_k^{-1},\end{align} which is again summable and hence allows the invocation of Borel-Cantelli's first lemma~\cite{bil95}, proving part (b).

\end{proof}

Next, we state a theorem that plays a crucial role in proving the overall convergence of ASTRO-DF iterates. Recall that even in deterministic derivative-free trust-region algorithms, unlike trust-region algorithms where derivative observations are available, the trust-region radius necessarily needs to decay to zero to ensure convergence. Theorem \ref{thm:delta-conv} states that this is indeed the case for ASTRO-DF. The proof rests on Lemma \ref{lem:fhat-bounded} and the assumed sufficient Cauchy decrease guarantee during Step \ref{ASTRO:TRsubprob} of Algorithm \ref{alg:mainlisting}. 

\begin{theorem}
\label{thm:delta-conv} Let Assumptions \ref{assump:diff-bddbelow}, \ref{assump:var}, \ref{assump:cauchy}, \ref{assump:hessian-bound} and \ref{assump:lambda_lbd} hold. Then $\Delta_{k}\xrightarrow{wp1}0$ as $k\rightarrow\infty$.
\end{theorem}

\begin{proof} We note during the $i$th iteration in Step \ref{ASTRO:construct} of Algorithm \ref{alg:mainlisting} that $\Fbar\left(\BFX_{i},N_{i}\right)$ and $\Fbar\left(\BFX_{i},\Ntilde_{i}\right)$ denote the function estimate at the point $\BFX_{i}$ \emph{before} entering and upon exiting {\tt{AdaptiveModelConstruction}} respectively.
\begin{equation}\label{eq:lim-fbar} \Fbar\left(\BFX_{k},N_{k}\right)=\Fbar\left(\BFX_{1},N_{1}\right)+\sum_{i=1}^{k-1}\left(A_{i}+B_{i}\right)\end{equation} where the summands $A_{i}=\Fbar\left(\BFX_{i+1},N_{i+1}\right)-\Fbar\left(\BFX_{i},\Ntilde_{i}\right)$ and $B_{i}=\Fbar\left(\BFX_{i},\Ntilde_{i}\right)-\Fbar\left(\BFX_{i},N_{i}\right)$. %%%% ADDED 2/2/16
In words $A_i$ represents the reduction in the function estimates during the $i$th iteration and $B_i$ represents the difference between the two estimates of the function at the point $\BFX_{i}$ at the end of iteration $i-1$ and $i$. 
%%%%%
We now make the following observations about $A_i$ and $B_i$. \begin{enumerate} \item[(a)] If $i$ is an unsuccessful iteration, then $A_i = 0$ since $\BFX_i= \BFX_{i+1}$. 
\item[(b)] If $i$ is a successful iteration, we know by definition that $\rhohat_{i}\geq\eta_{1}$. If we denote $\kappa_{efd}=(2\mu)^{-1}\eta_{1}\kappa_{fcd}\min\left\{ \left(\mu\kappa_{bhm}\right)^{-1},1\right\} $, then by Assumptions \ref{assump:cauchy} and \ref{assump:hessian-bound}, and by the assurance in Algorithm \ref{alg:sublisting} that $\Delta_{k}\leq\mu\left\|\nabla M_{k}\left(\BFX_{k}\right)\right\|$, we have
\begin{equation}\label{eq:lb_trust}\begin{aligned} A_{i} & \leq\eta_{1}\left(M_{i}\left(\BFX_{i+1}\right)-M_{i}\left(\BFX_{i}\right)\right)\\
 & \leq-\frac{\eta_{1}}{2}\kappa_{fcd}\left\| \nabla M_{i}\left(\BFX_{i}\right)\right\| \min\left\{ \frac{\left\| \nabla M_{i}\left(\BFX_{i}\right)\right\| }{\left\| \nabla^{2}M_{i}\left(\BFX_{i}\right)\right\| },\Delta_{i}\right\} \\
 & \leq-\kappa_{efd}\Delta_{i}^{2}.\end{aligned}
\end{equation}
%\item[(c)] For any given $c>0$, (\ref{eq:contra_lb}) in the proof of Lemma \ref{lem:fhat-bounded} assures us that $\mbP\left\{ \left|B_{i}\right|>c\mbox{ i.o.}\right\} \leq\mbP\left\{ \left|\Fbar\left(\BFX_{i+1}(\omega),N_{i+1}(\omega)\right)\right|+\left|\Fbar_{b}\left(\BFX_{i+1}(\omega),N_{i+1}(\omega)\right)\right|>c\mbox{ i.o.}\right\} =0.$ This implies that except for a set of measure zero, $\left | B_i\right | \leq c$ for large enough iteration $i$.
\item[(c)] For any given $c>0$, Lemma \ref{lem:fhat-bounded} ensures that $\mbP\left\{ \left|B_{i}\right|>c\ \mbox{i.o.}\right\}=0$ since \[\begin{aligned}
 \mbP\left\{ \left|\Fbar\left(\BFX_{i},\Ntilde_{i}\right)-\Fbar\left(\BFX_{i},N_{i}\right)\right|>c\right\} & \leq\mbP\left\{ \left|\Fbar\left(\BFX_{i},\Ntilde_{i}\right)-f\left(\BFX_{i}\right)\right|>\frac{c}{2}\right\} \\
 & \ +\mbP\left\{ \left|f\left(\BFX_{i}\right)-\Fbar\left(\BFX_{i},N_{i}\right)\right|>\frac{c}{2}\right\}, 
\end{aligned}
\]  using Boole's inequality (see Lemma \ref{lem:boole}). This implies that except for a set of measure zero, $\left | B_i\right | \leq c$ for large enough $i$.\end{enumerate}  

Now suppose $\mcD:=\left\{ \omega:\lim_{k\rightarrow\infty}\Delta_{k}(\omega)\neq0\right\}$ denotes the set of sample-paths for which the trust-region radius does not decay to zero. For contraposition, suppose $\mcD$ has positive measure. Consider a sample-path $\omega_0 \in \mcD$. Since unsuccessful iterations are necessarily contracting iterations, we can find $\delta(\omega_0)>0$ and the sub-sequence of successful iterations $\{k_j\}$ in the sample-path $\omega_0$ such that $\Delta_{k_j}(\omega_0) \geq \delta(\omega_0)$. This implies from observation (b) above that \begin{equation}\label{eq:Asuc} A_{k_j}(\omega_0) \leq -\kappa_{efd}\delta^2(\omega_0).\end{equation} Now, observation (a) above implies that \begin{equation}\label{eq:Aunsuc} A_{k_j+\ell}(\omega_0) \leq 0, \ell = 1, 2, \ldots, k_{j+1} - k_j -1.\end{equation} Also by the observation (c) above, and 
choosing $c = \frac{1}{3}\kappa_{efd}\delta^2(\omega_0)$, we see that for large-enough $i$, 
\begin{equation} \label{eq:Bi}\left|\Fbar\left(\BFX_{i}(\omega_{0}),\Ntilde_{i}(\omega_{0})\right)-\Fbar\left(\BFX_{i}(\omega_{0}),N_{i}(\omega_{0})\right)\right|\leq\frac{2}{3}\kappa_{efd}\delta^{2}(\omega_{0}).\end{equation} 
We then write for large-enough $j$,  
\begin{equation}\label{eq:cyclesum}\begin{aligned}\sum_{\ell=k_{j}}^{k_{j+1}-1}\left(A_{\ell}(\omega_{0})+B_{\ell}(\omega_{0})\right) & =A_{k_{j}}(\omega_{0})+\sum_{\ell=k_{j}}^{k_{j+1}-1}B_{\ell}(\omega_{0})\\
 & \leq A_{k_{j}}(\omega_{0})\\
 & \ \ \ \ +\Fbar\left(\BFX_{k_{j+1}-1}(\omega_{0}),\Ntilde_{k_{j+1}-1}(\omega_{0})\right)\\
 & \ \ \ \ -\Fbar\left(\BFX_{k_{j}+1}(\omega_{0}),N_{k_{j}+1}(\omega_{0})\right)\\
 & \leq-\frac{1}{3}\kappa_{efd}\delta^{2}(\omega_{0}),
\end{aligned}
\end{equation} 
where the first equality follows from observation (a) above, the first inequality follows from the definition of $B_{\ell}$, and the second inequality follows from (\ref{eq:Asuc}) and (\ref{eq:Bi}). The inequality in (\ref{eq:cyclesum}) (and the fact that there is an entire sequence $\{k_j\}$ of successful iterations) means that $\lim_{k \to \infty} \Fbar\left(\BFX_{k}\left(\omega_0\right),N_{k}(\omega_{0})\right) = -\infty$ thus contradicting Lemma \ref{lem:fhat-bounded}. The assertion of the theorem thus holds.\end{proof}

Relying on Theorem \ref{thm:delta-conv}, we now show that the model gradient converges to the true gradient almost surely. This, of course, does not imply that the true gradient itself converges to zero --- a fact that will be established subsequently. Implicit in the proof of Theorem \ref{thm:delta-conv} is the requirement that Algorithm \ref{alg:sublisting} terminates in finite time with probability one, a fact that we establish through Lemma \ref{lem:finite-improvement} in the Appendix.

\begin{lemma}
\label{lem:G-convergence} Let Assumptions \ref{assump:diff-bddbelow}, \ref{assump:var} -- \ref{assump:lambda_lbd} hold. Then $\left\| \nabla M_{k}\left(\BFX_{k}\right)-\nabla f\left(\BFX_{k}\right)\right\| \xrightarrow{wp1}0$ as $k\rightarrow\infty$.
\end{lemma}

\begin{proof} %We assume that the stochastic model $M_{k}$ constructed via Algorithm \ref{alg:sublisting} terminates in finite time. 
In Step \ref{AMC:sample_set} of Algorithm \ref{alg:sublisting}, $\Deltatilde_{k}w^{j_{k}-1}$ denotes the trust-region radius over which the model is constructed. (Note that due to Step \ref{AMC:update} of Algorithm \ref{alg:sublisting}, $\Deltatilde_{k}w^{j_{k}-1}$ may or may not equal the ending trust-region radius $\Delta_k$ upon completion of $k$ iterations of ASTRO-DF.) Then, we know from part (ii) of Lemma \ref{lem:model_error} that 
$$\left\| \nabla M_{k}\left(\BFX_{k}\right)-\nabla f\left(\BFX_{k}\right)\right\| \leq\kappa_{1}\left(\Deltatilde_{k}w^{j_{k}-1}\right)^{\theta}+\kappa_{2}\frac{\sqrt{\sum_{i=2}^{p}\left(E_{k,i}^{(j_{k})}-E_{k,1}^{(j_{k})}\right)^{2}}}{\left(\Deltatilde_{k}w^{j_{k}-1}\right)},$$ where 
%$E_{k,1}=\Fbar\left(\BFX_{k},N\left(\BFX_{k}\right)\right)-f\left(\BFX_{k}\right)$ is the error of the sampled function estimate at the center point of the trust-region, and 
$E_{k,i}^{(j_{k})}=\Fbar\left(\BFY_{i}^{(j_{k})},N\left(\BFY_{i}^{(j_{k})}\right)\right)-f\left(\BFY_{i}^{(j_{k})}\right)$ for $i=1,\ldots,p$ denotes the error due to sampling at point $\BFY_{i}^{(j_{k})}$ after the $j_{k}$th iteration of the contraction loop. (Recall that $\BFY_{1}^{(j_{k})}=\BFX_{k}$;  $p=d+1$ and $\theta=1$ in the linear interpolation models, and $p=(d+1)(d+2)/2$ and $\theta=2$ in the quadratic interpolation models. For the quantities $\kappa_1$ and $\kappa_2$ refer to part (ii) of Lemma \ref{lem:model_error}.)

From Theorem \ref{thm:delta-conv}, $\Delta_{k}\xrightarrow{wp1}0$ as $k \to \infty$, and hence, $\Deltatilde_{k}w^{j_{k}-1}\xrightarrow{wp1}0$ as $k \to \infty$. Also, $\sqrt{\sum_{i=2}^{p}\left(E_{k,i}^{(j_{k})}-E_{k,1}^{(j_{k})}\right)^{2}} \leq\sum_{i=2}^{p}\sqrt{\left(E_{k,i}^{(j_{k})}-E_{k,1}^{(j_{k})}\right)^{2}}=\sum_{i=2}^{p}\left|E_{k,i}^{(j_{k})}-E_{k,1}^{(j_{k})}\right|.$ Considering these two observations, it suffices to show that as $k \to \infty$, \begin{equation} \label{eq:mod_samp_err}\left(\Deltatilde_{k}w^{j_{k}-1}\right)^{-1}\sum_{i=2}^{p}\left|E_{k,i}^{(j_{k})}-E_{k,1}^{(j_{k})}\right| \xrightarrow{wp1}0.\end{equation} Towards this, we write for $c>0$, large enough $k$, some $\delta>0$, and by Boole's inequality (Lemma \ref{lem:boole}),
\begin{align}\label{eq:mod_err_bd} 
\mbP\left\{ \frac{\sum_{i=2}^{p}\left|E_{k,i}^{(j_{k})}-E_{k,1}^{(j_{k})}\right|}{\left(\Deltatilde_{k}w^{j_{k}-1}\right)}\geq c\right\} &\leq\sum_{i=2}^{p}\mathbb{E}\left[\mbP\left\{ \left|E_{k,i}^{(j_{k})}-E_{k,1}^{(j_{k})}\right|\geq \frac{c\left(\Deltatilde_{k}w^{j_{k}-1}\right)}{p-1}\,\vert\,\mcF_{k}\right\} \right]\nonumber \\
&\leq\sum_{i=2}^{p}\left(\mathbb{E}\left[\mbP\left\{ \left|E_{k,i}^{(j_{k})}\right|\geq\frac{c\left(\Deltatilde_{k}w^{j_{k}-1}\right)}{2(p-1)} \,\vert\,\mcF_{k}\right\}\right]\right.\nonumber \\
&\ \ \ \ \ \ \ \ \ \left.+\mathbb{E}\left[\mbP\left\{ \left|E_{k,1}^{(j_{k})}\right|\geq\frac{c\left(\Deltatilde_{k}w^{j_{k}-1}\right)}{2(p-1)} \,\vert\,\mcF_{k}\right\}\right]\right)\nonumber \\
&\leq2(p-1)^{3}4\left(c\Deltatilde_{k}w^{j_{k}-1}\right)^{-2}(1+\delta)\kappa_{ias}^{2}(\Deltatilde_{k}w^{j_{k}-1})^{4}\lambda_{k}^{-1}\nonumber \\
&\leq8(p-1)^{3}c^{-2}(1+\delta)\kappa_{ias}^{2}\Delta_{k}^{2}\lambda_{k}^{-1},
\end{align} where the second inequality above follows from the application of Boole's inequality (see Lemma~\ref{lem:boole}), and the penultimate inequality above follows from arguments identical to those leading to (\ref{eq:contra_lb}) in the proof of Lemma \ref{lem:fhat-bounded} after using the adaptive sample size expression in  (\ref{eq:inner-sampling-interpolation}). Since the right-hand side of (\ref{eq:mod_err_bd}) is summable, we can invoke the first Borel-Cantelli lemma~\cite{bil95} to conclude that (\ref{eq:mod_samp_err}) holds.
\end{proof}

We now show that for large enough iteration $k$, the steps within ASTRO-DF are always successful with probability one. This result is important in that it implies that the model gradient and the trust-region radius will remain in lock-step for large $k$, almost surely. The proof proceeds by dividing the model error into three components, each of which is shown to be controlled with probability one. 

\begin{theorem}
\label{thm:finite-unsuccess} Let Assumptions \ref{assump:diff-bddbelow} -- \ref{assump:lambda_lbd} hold. Then $\mbP\left\{ \rhohat_{k}<\eta_1, \rm{ i.o.}\right\} =0$ for any $\eta_1 \in\left(0,1\right)$.
\end{theorem}

\begin{proof}
%Note that constants $\kappa_{fcd}$ associated with factor of Cauchy decrease in Assumption \ref{assump:cauchy} and $\kappa_{bhm}$ associated with the upper bound of the model Hessian in Assumption \ref{assump:hessian-bound} are not fixed. 
At the end of Step \ref{AMC:end_loop} of Algorithm \ref{alg:sublisting}, let $m_{k}^{(j_{k})}\left(\BFz\right)$ be the interpolation model of $f$ constructed on the poised set $\mcY_{k}$. (Of course, we cannot construct $m_k(\cdot)$ explicitly because the true function values are unknown.)
%$=\mathbb{E}\left[M_{k}^{\left(j\right)}\left(\BFz\right)\right]$. 
Then $m_{k}^{(j_{k})}\left(\BFz\right)$ is a $\left(\kappa_{ef},\kappa_{eg}\right)$-fully-linear model of $f$ on $\mcB\left(\BFX_{k};\Deltatilde_{k}w^{j_{k}-1}\right)$ and since  $\Delta_{k}\geq\Deltatilde_{k}w^{j_{k}-1}$, by the Lemma in \cite[p. 200]{conschvic2009} we have that $m_{k}\left(\cdot\right)$ is a $\left(\kappa_{ef},\kappa_{eg}\right)$-fully-linear model of $f$ on $\mcB\left(\BFX_{k};\Delta_{k}\right)$. In addition, Algorithm \ref{alg:sublisting} ensures that $\Delta_{k}\leq\mu\left\| \nabla M_{k}\left(\BFX_{k}\right)\right\|$.

Assumption \ref{assump:cauchy} on the Cauchy decrease in the minimization problem implies that
\begin{align}\label{eq:mindecrease} M_{k}\left(\BFX_{k}\right)-M_{k}\left(\BFXtilde_{k+1}\right) & \geq\frac{\kappa_{fcd}}{2}\left\| \nabla M_{k}\left(\BFX_{k}\right)\right\| \min\left\{ \frac{\left\| \nabla M_{k}\left(\BFX_{k}\right)\right\| }{\left\| \nabla^2 M_{k}\left(\BFX_{k}\right)\right\| },\Delta_{k}\right\} \nonumber \\
 & \geq\frac{\kappa_{fcd}}{2}\left\| \nabla M_{k}\left(\BFX_{k}\right)\right\| \min\left\{ \frac{\Delta_{k}}{\mu\kappa_{bhm}},\Delta_{k}\right\} \nonumber \\
 & \geq \kappa_{md} \Delta_{k}^2.
% & =\frac{\kappa_{fcd}}{2}\min\left\{ \frac{1}{\mu\kappa_{bhm}},1\right\} \left\| \nabla M_{k}\left(\BFX_{k}\right)\right\| \Delta_{k} \geq\kappa_{md}\Delta_{k}^{2}.
\end{align} where $\kappa_{md}=(2\mu\kappa_{bhm})^{-1}\mbox{min}(\mu\kappa_{bhm}^{-1},1)\kappa_{fcd}$. %Let $\eta=\eta_1$ be the one chosen in Step 0 of Algorithm \ref{alg:mainlisting}. 
Recall that $$\rhohat_{k} \defn \frac{\Fbar\left(\BFX_{k},\Ntilde_{k}\right)-\Fbar\left(\BFXtilde_{k+1},\Ntilde_{k+1}\right)}{M_{k}(\BFX_{k})-M_{k}(\BFXtilde_{k+1})}$$ and that $\Fbar\left(\BFX_{k},\Ntilde_{k}\right)=M_{k}\left(\BFX_{k}\right)$. Now using Boole's inequality (see Lemma \ref{lem:boole}) and (\ref{eq:mindecrease}), we can write
\begin{align}\label{eq:error_split} \mbP\left\{ \rhohat_{k}<\eta_{1}\right\} &\leq\mbP\left\{ \left|1-\rhohat_{k}\right|\geq1-\eta_{1}\right\} \nonumber \\
% & \leq\mbP\left\{ \frac{\left|\Fbar\left(\BFXtilde_{k+1},\Ntilde_{k+1}\right)-M_{k}\left(\BFXtilde_{k+1}\right)\right|+\left|\Fbar\left(\BFX_{k},N_{k}\right)-M_{k}\left(\BFX_{k}\right)\right|}{\left|M_{k}(\BFX_{k})-M_{k}(\BFXtilde_{k+1})\right|}\geq1-\eta_{1}\right\} \nonumber \\
 &\leq\mbP\left\{ \left|\Fbar\left(\BFXtilde_{k+1},\Ntilde_{k+1}\right)-M_{k}\left(\BFXtilde_{k+1}\right)\right|\geq\left(1-\eta_{1}\right)\kappa_{md}\Delta_{k}^{2}\right\} \nonumber \\
 &\leq \mbP\left\{ Err_{1} \geq\eta'\Delta_{k}^{2}\right\} +\mbP\left\{ Err_{2}\geq\eta'\Delta_{k}^{2}\right\} + \mbP\left\{ Err_{3} \geq\eta'\Delta_{k}^{2}\right\},\end{align} where $Err_{1} \defn \left|\Fbar\left(\BFXtilde_{k+1},\Ntilde_{k+1}\right)-f\left(\BFXtilde_{k+1}\right)\right|$, $Err_{2} \defn \left|f\left(\BFXtilde_{k+1}\right)-m_{k}\left(\BFXtilde_{k+1}\right)\right|$, $Err_{3} \defn \left|m_{k}\left(\BFXtilde_{k+1}\right)-M_{k}\left(\BFXtilde_{k+1}\right)\right|$, and $\eta'=3^{-1}\left(1-\eta_{1}\right)\kappa_{md}$. In what follows, we establish $\mbP\left\{ \rhohat_{k}<\eta_{1}\:\mathrm{ i.o.}\right\} =0$ by demonstrating that each of the errors $Err_{1}$, $Err_{2}$ and $Err_{3}$ exceeding $\eta'\Delta_{k}^{2}$ infinitely often has probability zero.

We first analyze the stochastic sampling error probability
$\mbP\left\{ Err_{1} \geq\eta'\Delta_{k}^{2}\right\} $ appearing on the right-hand side of (\ref{eq:error_split}). Using arguments identical to those leading to (\ref{eq:contra_lb}) in the proof of Lemma \ref{lem:fhat-bounded}, it is seen that \begin{equation}\label{eq:err1} \mbP\left\{ \left|\Fbar\left(\BFXtilde_{k+1},\Ntilde_{k+1}\right)-f\left(\BFXtilde_{k+1}\right)\right|\geq\eta'\Delta_{k}^{2}\ \ \mathrm{i.o.}\right\} =0.\end{equation} 

Next we analyze the deterministic model error probability $\mbP\left\{ Err_2 \geq\eta'\Delta_{k}^{2}\right\} $ appearing on the right-hand side of (\ref{eq:error_split}). Since we know from the postulates of the theorem that $m_k\left(\BFz\right)$ is a $\left(\kappa_{ef},\kappa_{eg}\right)$-fully-linear model of $f$ on $\mcB\left(\BFX_{k};\Delta_{k}\right)$, implying that if $\eta_1$ is chosen so that $\eta' = \frac{1}{3}(1-\eta_1)\kappa_{md} >\kappa_{ef}$, we have \begin{equation}\label{eq:err2}\mbP\left\{ \left|f\left(\BFXtilde_{k+1}\right)-m_{k}\left(\BFXtilde_{k+1}\right)\right|\geq\eta'\Delta_{k}^{2}\ \ \mathrm{ i.o.}\right\}=0.\end{equation}

Finally, we analyze the stochastic interpolation error probability $\mbP\left\{ Err_{3} \geq\eta'\Delta_{k}^{2}\right\} $ appearing on the right-hand side of (\ref{eq:error_split}). %Recall $p=d+1$ for linear interpolation. 
Using part (i) of Lemma \ref{lem:model_error} and relabeling $\BFX_{k}$ to 
$\BFY_{1}$ for readability, we write
\begin{align}\label{eq:stoch_error}\mbP\left\{ Err_{3}>\eta' \Delta_k^2 \right\}  & \leq \mbP\left\{ \max_{\substack{\BFY_{i}\in\mcY_{k},\\i=1,2,\ldots,p}}\left|\Fbar\left(\BFY_{i},N\left(\BFY_{i}\right)\right)-f\left(\BFY_{i}\right)\right|>\frac{\eta'\Delta_k^2}{p\Lambda}\right\} \nonumber \\
& \leq\sum_{i=1}^{p}\mbP\left\{ \left|\Fbar\left(\BFY_{i},N\left(\BFY_{i}\right)\right)-f\left(\BFY_{i}\right)\right|>\frac{\eta' \Delta_k^2}{p^2\Lambda} \right\} \nonumber \\
& = \sum_{i=1}^{p}\mathbb{E}\left[\mbP\left\{ \left|\Fbar\left(\BFY_{i},N\left(\BFY_{i}\right)\right)-f\left(\BFY_{i}\right)\right|>\frac{\eta' \Delta_k^2}{p^2\Lambda} \, \vert \, \mcF_k\right\}\right].
\end{align} 
Now using (\ref{eq:stoch_error}) and arguments identical to those leading to (\ref{eq:contra_lb}) in the proof of Lemma \ref{lem:fhat-bounded} (and the sample size expression (\ref{eq:inner-sampling-interpolation}) in Step \ref{AMC:estimate} of Algorithm \ref{alg:sublisting}), we can then say for large enough $k$ and some $\delta > 0$ that \begin{align}\label{eq:err3_bound}\mbP\left\{ Err_3>\eta' \Delta_k^2 \right\} & \leq \frac{p^{5}\Lambda}{\lambda_k(\eta'\Delta_k^{2})^{2}}(1+\delta)\kappa_{ias}^2\Delta_k^4 \nonumber \\
& \leq \frac{p^{5}\Lambda}{\lambda_k\eta'^{2}}(1+\delta)\kappa_{ias}^2.
\end{align} Since $\lambda_k$ is chosen so that $k^{1+\epsilon} = \mcO(\lambda_k)$ for some $\epsilon > 0$, we see that (\ref{eq:err3_bound}) implies that $\mbP\left\{ \left(\Fbar\left(\BFY_{i},N\left(\BFY_{i}\right)\right)-f\left(\BFY_{i}\right)\right)>\eta'\Delta_{k}^{2} \textrm{ i.o.} \right\} = 0$ by Borel-Cantelli. This in turn implies from (\ref{eq:stoch_error}) that \begin{equation}\label{eq:err3}\mbP\left\{ \left|M_{k}\left(\BFXtilde_{k+1}\right)-m_{k}\left(\BFXtilde_{k+1}\right)\right|\geq\eta'\Delta_{k}^{2} \textrm{ i.o.} \right\} = 0.\end{equation}

Conclude from (\ref{eq:err1}), (\ref{eq:err2}), and (\ref{eq:err3}) that each of the errors $Err_{1}$, $Err_{2}$, and $Err_{3}$ exceeding $\eta'\Delta_k^2$ infinitely often has probability zero and the assertion of Theorem \ref{thm:finite-unsuccess} holds.
\end{proof}

\begin{lemma}
\label{lem:bounds} 
%If there exists a constant $\kappa_{lbg}>0$, such that $\left\| \nabla M_{k}\left(\BFX_{k}\right)\right\| \geq\kappa_{lbg}$ for large $k$, then there exists a constant $\kappa_{lbd}>0$ such that $\Delta_{k}\geq\kappa_{lbd}$ for large $k$ with probability one.
%% EDIT: 7/13/16
For any sample path $\omega\in\Omega$ if there exists a constant $\kappa_{lbg}(\omega)>0$, such that $\left\| \nabla M_{k}\left(\BFX_{k}(\omega)\right)\right\| \geq\kappa_{lbg}(\omega)$ for large enough $k$, then there exists a constant $\kappa_{lbd}(\omega)>0$ such that $\Delta_{k}(\omega)\geq\kappa_{lbd}(\omega)$ for large enough $k$.
\end{lemma}

\begin{proof}
Let $K_{g}(\omega)>0$ be such that $\left\| \nabla M_{k}\left(\BFX_{k}(\omega)\right)\right\| \geq\kappa_{lbg}$ if $k> K_{g}(\omega)$. From Theorem \ref{thm:finite-unsuccess}, we let $K_{s}(\omega)>0$ be such that $K_{s}(\omega)-1$ is the last unsuccessful iteration, that is, $k$ is a successful iteration if $k\geq K_{s}(\omega)$. Then $\Deltatilde_{k}(\omega)>\Delta_{k-1}(\omega)$ for all $k\geq K_{s}(\omega)$. For $k\geq \max\left\{ K_{g}(\omega),K_{s}(\omega)\right\}+1 $, consider the two cases below when Algorithm \ref{alg:sublisting} starts.

\begin{itemize}
\item[Case 1:] ($\Deltatilde_{k}(\omega)\geq\mu\left\| \nabla M_{k}\left(\BFX_{k}(\omega)\right)\right\|$): Since $\Deltatilde_{k}(\omega)\geq\mu\left\| \nabla M_{k}\left(\BFX_{k}(\omega)\right)\right\|$, the inner loop of Algorithm \ref{alg:sublisting} is executed, implying that \[\Delta_k(\omega) \geq \beta\|\nabla M_{k}\left(\BFX_{k}(\omega)\right)\| \geq \beta \kappa_{lbg}(\omega).\]
\item[Case 2:] ($\Deltatilde_{k}(\omega) < \mu\left\| \nabla M_{k}\left(\BFX_{k}(\omega)\right)\right\|$): In this scenario, the inner loop of Algorithm \ref{alg:sublisting} is not executed, implying that $\Delta_{k}(\omega)=\Deltatilde_{k}(\omega)=\gamma_{1}\Delta_{k-1}(\omega)$ meaning that the trust-region radius expands from the previous iteration.
\end{itemize}

Case 1 and Case 2 iterations are mutually exclusive and collectively exhaustive. Case 1 iterations imply, under the assumed postulates, that $\Delta_k(\omega) \geq \beta \kappa_{lbg}(\omega)$; Case 2 iterations result in an expanded trust-region radius. Conclude from these assertions that $\Delta_{k}(\omega)\geq \min\left\{\beta \kappa_{lbg}(\omega), \Delta_{\max\left\{K_g(\omega),K_s(\omega)\right\}}\right\}$.
\end{proof}

An important observation from the Algorithms \ref{alg:mainlisting} and \ref{alg:sublisting} is that the difference between the function estimates of two consecutive iterates can be increasing; or in other words $\Fbar\left(\BFX_k\right)$ is not necessarily monotone decreasing. When iteration $k$ is unsuccessful, that is, $\BFX_{k}=\BFX_{k+1}$, it is possible that $\Fbar\left(\BFX_{k},N_{k}\right)<\Fbar\left(\BFX_{k+1},N_{k+1}\right)$. When iteration $k$ is successful, it must be true that $\Fbar\left(\BFX_{k},\Ntilde_{k}\right)>\Fbar\left(\BFX_{k+1},N_{k+1}\right)$ but it is still possible that $\Fbar\left(\BFX_{k},N_{k}\right)<\Fbar\left(\BFX_{k+1},N_{k+1}\right)$ since $\Fbar\left(\BFX_{k},N_{k}\right)\neq\Fbar\left(\BFX_{k},\Ntilde_{k}\right)$. 
We are now fully setup to demonstrate that ASTRO-DF's iterates converge to a first-order critical point with probability one. 

\begin{theorem}\label{thm:conv}
Let Assumptions \ref{assump:diff-bddbelow} -- \ref{assump:lambda_lbd} hold. Then $\left\| \nabla f\left(\BFX_{k}\right)\right\| \xrightarrow{wp1}0$ as $k \to \infty$.
\end{theorem}

\begin{proof} We start by making the following two observations. 
\begin{enumerate} \item[(a)] Lemma \ref{lem:bounds} and Theorem \ref{thm:delta-conv} together imply that $\lim\inf_{k\rightarrow\infty}\left\| \nabla M_{k}\left(\BFX_{k}\right)\right\| =0$. This, along with Lemma \ref{lem:G-convergence}, imply $\lim\inf_{k\rightarrow\infty}\left\| \nabla f\left(\BFX_{k}\right)\right\| =0$ almost surely. 

\item[(b)] Consider the sequence of function estimates $\{\Fbar\left(\BFX_{k},N_{k}\right)\}$. Since we know from Theorem \ref{thm:finite-unsuccess} that iterations are ultimately successful with probability one, we see that $\{\Fbar\left(\BFX_{k},N_{k}\right)\}$ is non-increasing for large enough $k$ with probability one. This and the fact that the sequence $\{\Fbar\left(\BFX_{k},N_{k}\right)\}$ is bounded (by Lemma \ref{lem:fhat-bounded}) with probability one implies that the difference $\Fbar\left(\BFX_{k+1},N_{k+1}\right) - \Fbar\left(\BFX_{k},N_{k}\right) \xrightarrow{wp1}0.$ \end{enumerate}       

Suppose the assertion of Theorem \ref{thm:conv} is not true. The there exists a set $\mathcal{D}$ of positive measure such that for any sample-path $\omega \in \mathcal{D}$, there exists a subsequence of iterations $\left\{ t_{i}\right\} $ satisfiying $\left\Vert \nabla f\left(\BFX_{t_{i}}\right)\right\Vert >3\epsilon$ for some $\epsilon>0$. (In the previous statement and in what follows, we have suppressed $\omega$ from the notation for convenience.) Due to our observation in (a), corresponding to each element $t_{i}$, there exists $\ell_{i}=\ell\left(t_{i}\right)$, the first iteration after $t_{i}$, such that $\left\Vert \nabla f\left(\BFX_{\ell_{i}}\right)\right\Vert <2\epsilon$. Therefore if $\mcK_{i}=\left\{k:t_{i}\leq k\leq \ell_{i}\right\}$, then $\left\Vert \nabla f\left(\BFX_{k}\right)\right\Vert\geq2\epsilon$ for all $k\in\mcK_{i}$. 

Now choose $i$ large enough so that by Theorem \ref{thm:finite-unsuccess}, Lemma \ref{lem:G-convergence}, Theorem \ref{thm:delta-conv}, and Lemma \ref{lem:fhat-bounded} for all $k\in\mcK_{i}$, (i) $\rhohat_{k}\geq\eta_{1}$ (only successful iterations), (ii) $\left\Vert \nabla M_{k}\left(\BFX_{k}\right)\right\Vert \geq \epsilon$ (model gradient close to the function gradient), (iii) $\Delta_{k}\leq\kappa_{bhm}^{-1}\epsilon$ (trust-region radius small), and (iv) $\left|\Fbar\left(\BFX_{k},N_{k}\right)-f\left(\BFX_{k}\right)\right|\leq 8^{-1}\eta_{1}\epsilon\kappa_{fcd}\Delta_{k}$ (simulation error small). As a result and by Cauchy reduction in Assumption \ref{assump:cauchy} we get  
\begin{align}
\Fbar\left(\BFX_{k+1},N_{k+1}\right)-\Fbar\left(\BFX_{k},\Ntilde_{k}\right) & \leq-2^{-1}\eta_1\kappa_{fcd}\left\|\nabla M_k\left(\BFX_k\right)\right\|\min\left\{\kappa_{bhm}^{-1}\left\|\nabla M_k\left(\BFX_k\right)\right\|,\Delta_k\right\}\nonumber\\
 & \leq-2^{-1}\eta_1\kappa_{fcd}\epsilon\min\left\{\kappa_{bhm}^{-1}\epsilon,\Delta_k\right\}\nonumber\\
 & =-\eta_{1}\frac{\kappa_{fcd}}{2}\epsilon\Delta_{k}.\label{eq:fred} 
 \end{align}
Therefore
\begin{align}
\Fbar\left(\BFX_{k+1},N_{k+1}\right)-\Fbar\left(\BFX_{k},N_{k}\right) & \leq\Fbar\left(\BFX_{k+1},N_{k+1}\right)-\Fbar\left(\BFX_{k},\Ntilde_{k}\right)\nonumber\\
 & +\left|\Fbar\left(\BFX_{k},\Ntilde_{k}\right)-f\left(\BFX_{k}\right)\right|+\left|f\left(\BFX_{k}\right)-\Fbar\left(\BFX_{k},N_{k}\right)\right|\nonumber\\
 & \leq-\eta_{1}\frac{\kappa_{fcd}}{2}\epsilon\Delta_{k}+\eta_{1}\frac{\kappa_{fcd}}{4}\epsilon\Delta_{k}=-\eta_{1}\frac{\kappa_{fcd}}{4}\epsilon\Delta_{k},\label{eq:fdiff} 
 \end{align}
 and as a result $\Delta_{k}  \leq -4(\eta_{1}\kappa_{fcd}\epsilon)^{-1}\left(\Fbar\left(\BFX_{k+1},N_{k+1}\right)-\Fbar\left(\BFX_{k},N_{k}\right)\right)$ for all $k\in\mcK_{i}$. It follows that \begin{align}\left\Vert \BFX_{\ell_{i}}-\BFX_{t_{i}}\right\Vert  & \leq\sum_{j\in\mcK_{i}}\left\Vert \BFX_{j+1}-\BFX_{j}\right\Vert \leq\sum_{j\in\mcK_{i}}\Delta_{j}\nonumber\\
 & \leq\frac{-4}{\eta_{1}\kappa_{fcd}\epsilon}\sum_{j\in\mcK_{i}}\Fbar\left(\BFX_{j+1},N_{j+1}\right)-\Fbar\left(\BFX_{j},N_{j}\right)\nonumber\\
 & \leq\frac{-4}{\eta_{1}\kappa_{fcd}\epsilon}\left(\Fbar\left(\BFX_{\ell_{i}},N_{\ell_{i}}\right)-\Fbar\left(\BFX_{t_{i}},N_{t_{i}}\right)\right).
\label{eq:xdiff}  \end{align}
The inequality in (\ref{eq:xdiff}) and our observation in (b) imply that \begin{equation}\label{cauchydiff}\left\Vert \BFX_{\ell_{i}}-\BFX_{t_{i}}\right\Vert \to 0 \mbox{ as } i \to \infty. \end{equation} 
Furthermore, since \[\begin{aligned}\left|f\left(\BFX_{\ell_{i}}\right)-f\left(\BFX_{t_{i}}\right)\right| \leq & \left|\bar{F}\left(\BFX_{\ell_{i}},N_{\ell_{i}}\right)-\bar{F}\left(\BFX_{t_{i}},N_{t_{i}}\right)\right| \\ & + \left|f\left(\BFX_{\ell_{i}}\right)-\bar{F}\left(\BFX_{\ell_{i}},N_{\ell_{i}}\right)\right|
+ \left|f\left(\BFX_{t_{i}}\right)-\bar{F}\left(\BFX_{t_{i}},N_{t_{i}}\right)\right|,\end{aligned}\] we see that \begin{equation}\label{cauchyfndiff} \left|f\left(\BFX_{\ell_{i}}\right)-f\left(\BFX_{t_{i}}\right)\right| \to 0 \mbox{ as } i \to \infty.\end{equation}

Using (\ref{cauchydiff}) and (\ref{cauchyfndiff}), and since the gradient $\nabla f(\BFx)$ is Lipschitz continuous, we conclude $\left\Vert \nabla f\left(\BFX_{\ell_{i}}\right)-\nabla f\left(\BFX_{t_{i}}\right)\right\Vert \to 0$  as $i \to \infty$. This, however, gives us a contradiction since the definition of $t_i$ and $\ell_i$ dictate that $\left\Vert \nabla f\left(\BFX_{t_{i}}\right)\right\Vert >3\epsilon$ and $\left\Vert \nabla f\left(\BFX_{\ell_{i}}\right)\right\Vert < 2\epsilon$.\end{proof}

\section{FURTHER REMARKS AND DISCUSSION}
\label{sec:final}

Over the last decade or so, derivative-free trust-region  algorithms have deservedly enjoyed great attention and success in the deterministic optimization context. Analogous algorithms for the now widely prevalent and important Monte Carlo stochastic optimization context, where only stochastic function oracles are available, is poorly studied. This paper develops adaptive sampling trust-region optimization derivative-free algorithms (called ASTRO-DF) for solving low to moderate dimensional stochastic optimization problems. The key idea within ASTRO-DF is to endow a derivative-free trust-region algorithm with an adaptive sampling strategy for function estimation. The  extent of such sampling at a visited point depends on the estimated proximity of the point to a solution, calculated by balancing the estimated standard error of the function estimate with a certain power of the incumbent trust-region radius. So, just as one might expect of efficient algorithms, Monte Carlo sampling in ASTRO-DF tends to be low during the early iterations compared to later iterations, when the visited points are more likely to be closer to a first-order critical point. More importantly, however, the schedule of sampling is not predetermined (as in most stochastic approximation and sample-average approximation algorithms) but instead adapts to the prevailing algorithm trajectory and the needed precision of the function estimates. 

We show that ASTRO-DF's iterates exhibit global convergence to a first-order critical point with probability one. While the proofs are detailed, convergence follows from two key features of ASTRO-DF: (i) the stochastic interpolation models are constructed across iterates in such a way that the error in the stochastic interpolation model is guaranteed to remain in lock-step with (a certain power of) the trust-region radius;  and (ii) the optimization within the trust-region step is performed in such a way as to guarantee Cauchy reduction and then the objective function is evaluated at the resulting candidate point. Remarkably, the features (i) and (ii) together ensure that the sequence of trust-region radii necessarily need to converge to zero with probability one, and that the model gradient, the true gradient and the trust-region radius all have to remain in lock-step, thus guaranteeing convergence to a first-order critical point with probability one. The key driver for efficiency is adaptive sampling, making all sample sizes within ASTRO-DF stopping times that are explicitly dependent on algorithm trajectory.

Four other points are worthy of mention.

\begin{enumerate} \item[(i)] Our proofs demonstrate global convergence to first-order critical points. Corresponding proofs of convergence to a second-order critical point can be obtained in an identical fashion by driving some measure of second-order stationarity to zero instead of the model gradient. \item[(ii)] The adaptive sampling ideas and the ensuing proofs we have presented in this paper are for the specific case of stochastic interpolation models. It seems to us, however, that the methods of proof presented in this paper can be co-opted (with care) into other potentially more powerful model construction ideas such as regression~\cite{billups2013derivative} and kriging~\cite{anknelsta2010,stein1}. Which of such ideas result in the best derivative-free trust-region algorithms (as measured by practical performance and asymptotic efficiency) remains to be seen. \item[(iii)] We have presented no proof that ASTRO-DF's iterates achieve the Monte Carlo canonical rate~\cite{asmussen2007stochastic}. Demonstrating that ASTRO-DF's iterates achieve the canonical rate will rely on rate results for derivative-free trust-region algorithms in the deterministic context, some of which are only now appearing~\cite{garjudvic2015}. We speculate, however, that ASTRO-DF's iterates do enjoy the canonical rate, as our analogous work~\cite{hashemi2014adaptive} in a different context has demonstrated. \item[(iv)] The asymptotic sampling rate within ASTRO-DF is approximately $\mcO\left(\Delta_k^{-4}\right)$, where $\Delta_k$ is the incumbent trust-region radius. (See (\ref{mukghobd}) in the proof of Lemma \ref{lem:fhat-bounded}.) This sampling stipulation is comparable to that prescribed in two other prominent recent studies~\cite{chemensch2015,lar2012}.The $\mcO(\Delta^{-4})$ sampling appears to be the minimum needed to guarantee convergence in derivative-free trust-region methods \emph{without assumptions on the tail-behavior of the error driving simulation observations.} A question of interest is whether the $\mcO(\Delta^{-4})$ sampling stipulation can be relaxed by assuming that the simulation error is light-tailed, that is, the errors have a well-defined moment-generating function. Such assumption will likely allow modifying the results in \cite{chemensch2015} through the use of the Chernoff bound instead of the Chebyshev inequality, leading to a weakening of the sampling stipulation. The corresponding modification in ASTRO-DF will involve proving a variation of Theorem 3.7 which we believe will be a contribution in itself. Assuming that the simulation errors are light-tailed is reasonable. For instance, many distributions we see in practice, e.g, normal, gamma, beta, are light-tailed; any distribution with bounded support is light-tailed. \end{enumerate}

\appendix
%\section{Appendix}
%\appendix
%\section{Proof of Theorem \ref{thm:meanasylim}}

%\noindent Since $\sigmahat^2_n \xrightarrow{wp1} \sigma^2$ as $n \to \infty$, it is clear that the assertion (i) in the statement of Theorem \ref{thm:meanasylim} holds. 

%To prove the assertion in (iii), denoting $S_N = \sum_{j=1}^N X_i$, we see that \begin{align}\label{stoppingtail} \mbP\{\bar{X}_N > t\sigma\} & \leq \mbP\{\bar{S}_N > t\sigma \lambda\} \nonumber \\ &= \mbP\{S^2_N > t^2\sigma^2 \lambda^2\} \nonumber \\ \leq \frac{\mathbb{E}[S_N^2]}{t^2 \sigma^2 \lambda^2},\end{align} where the first inequality above follows by the definition of $N$, the equality follows since $t\sigma > 0$, and the second inequality through Markov~\cite{bil95}. However, we know from (ii) that $S_N$ satisfies the postulates of Wald's first and second lemmas~\cite{wald}, and therefore $\mathbb{E}[S_N^2] = \sigma^2\mathbb{E}[N]$. Plugging this in (\ref{stoppingtail}), we get \begin{equation} \label{stoppingtail2}  \mbP\{\bar{X}_N > t\sigma\} \leq \frac{\mathbb{E}[N]}{t^2 \lambda^2}.\end{equation} 

\section{Proof of part (i) of Lemma \ref{lem:model_error}}
\noindent We know that for all $\BFz\in\mcB\left(\BFY_{1};\Delta\right)$, \[m\left(\BFz\right)=\sum_{i=1}^{p}\ell_{i}\left(\BFz\right)f\left(\BFY_{i}\right);\ M\left(\BFz\right)=\sum_{i=1}^{p}\ell_{i}\left(\BFz\right)\Fbar\left(\BFY_{i},n(\BFY_{i})\right),\] where $\ell_{j}\left(\BFz\right)$ are the Lagrange polynomials associated with the set $\mcY$. Since $\mcY$ is $\Lambda$-poised in $\mcB\left(\BFY_{1};\Delta\right)$, we know (see Chapter 3 in~\cite{conschvic2009}) that \begin{equation}\label{lambdapoised_lagrange} \Lambda \geq \Lambda_{\ell} = \max_{i=1,2,\ldots,p}\max_{\BFz\in\mcB\left(\BFY_{1};\Delta\right)}\left|\ell_{i}\left(\BFz\right)\right|.\end{equation} Now write, for $\BFz \in \mcB\left(\BFY_{1};\Delta\right)$,
\[\begin{aligned}\left|M\left(\BFz\right)-m\left(\BFz\right)\right| & =\left|\sum_{i=1}^{p}\ell_{i}\left(\BFz\right)\left(\Fbar\left(\BFY_{i},n(\BFY_{i})\right)-f\left(\BFY_{i}\right)\right)\right|\\
% & \leq\Lambda_{\ell}\sum_{i=1}^{p}\left|\Fbar\left(\BFY_{i},N\left(\BFY_{i}\right)\right)-f\left(\BFY_{i}\right)\right|\\
 & \leq p\Lambda_{\ell}\max_{i\in\left\{ 1,2,\ldots,p\right\} }\left|\Fbar\left(\BFY_{i},n(\BFY_{i})\right)-f\left(\BFY_{i}\right)\right| \\
& \leq p\Lambda\max_{i\in\left\{ 1,2,\ldots,p\right\} }\left|\Fbar\left(\BFY_{i},n(\BFY_{i})\right)-f\left(\BFY_{i}\right)\right|,
\end{aligned}
\]
where the last inequality follows from (\ref{lambdapoised_lagrange}).

\section{Proof of part (ii) of Lemma \ref{lem:model_error}}

%\begin{proof} 

If the model $M(\cdot)$ is a stochastic linear interpolation model, we see that for $i = 1, 2, 3, \ldots, p$ \begin{equation}\label{lmodelerr}\left(\BFY_{i}-\BFY_{1}\right)^{T}\nabla M\left(\BFY_{1}\right)=M\left(\BFY_{i}\right)-M\left(\BFY_{1}\right)= f\left(\BFY_{i} \right)-f\left(\BFY_{1}\right)+ E_{i}-E_{1}.\end{equation} Now re-trace the steps of the proof on pages 26 and 27 of~\cite{conschvic2009}, while carrying the additional term $E_{i}-E_{1}$ appearing on the right-hand side of (\ref{lmodelerr}).

If the model $M(\cdot)$ is a stochastic quadratic interpolation model, we write $\BFz \in\mcB(\BFY_{1}, \Delta)$, $M(\BFz) = c + \BFz^Tg + \frac{1}{2}\BFz^T H\BFz = f(\BFz) + e^f(\BFz); \nabla M(\BFz) = H\BFz + g = \nabla f(\BFz) + e^g(\BFz); \nabla^2 M(\BFz) = H = \nabla^2 f(\BFz) + e^H(\BFz)$. Now write, after subtracting the expression for $M(\BFz)$ from that for $M(\BFY_i), i = 1, \ldots, p+1$ to get \begin{equation}\label{qmodelerr} (\BFY_i - \BFz)^Tg + \frac{1}{2}(\BFY_i - \BFz)^TH(\BFY_i - \BFz) + (\BFY_i - \BFz)^TH\BFz = f(\BFY_i) -f(\BFz) + E^i -e^f(\BFz).\end{equation} Notice that the equation in~(\ref{qmodelerr}) is identical to the corresponding equation on page 53 of~\cite{conschvic2009} except for the extra term $E^i$ appearing on the right-hand side of (\ref{qmodelerr}). Re-trace the steps on pages 53, 54, and 55 of~\cite{conschvic2009}.
%\end{proof}

\section{Model Construction Algorithm Termination} In what follows, we demonstrate through the following result that the model construction algorithm (Algorithm \ref{alg:sublisting}) terminates with probability one, whenever the incumbent solution $\BFX_{k}$ is not a first-order critical point.

%This guarantee is important to ensure that the algorithm does succeed in creating a model of the specified quality in a finite number of steps, unless the current iterate is a first-order critical point. 

\begin{lemma}
\label{lem:finite-improvement} 
%If there is a constant  $C>0$ such that $\left\| \nabla f\left(\BFX_{k}(\omega)\right)\right\| >C$ almost surely, then 
Suppose the incumbent solution $\BFX_{k}\in\real^d$ during the $k$th iteration is not first-order critical, that is, $\nabla f(\BFX_{k}) \neq 0.$ Then Algorithm \ref{alg:sublisting} terminates in a finite number of steps with probability one.
\end{lemma}

\begin{proof} Set $\left\|\nabla f(\BFX_{k})\right\| = c' \neq 0$. We will prove the assertion through a contradiction argument. 

First, we notice that the contraction loop (Steps \ref{AMC:sample_set}--\ref{AMC:end_loop}) in Algorithm \ref{alg:sublisting} is not entered if $\mu\left\| \nabla M_{k}\left(\BFX_{k}\right)\right\| \geq\Deltatilde_{k}$, in which case Algorithm \ref{alg:sublisting} terminates trivially. 

Next, suppose $\mu\left\| \nabla M_{k}\left(\BFX_{k}\right)\right\| < \Deltatilde_{k}$ and that the contraction loop in Steps \ref{AMC:sample_set}--\ref{AMC:end_loop} of Algorithm \ref{alg:sublisting} is infinite. Let $\nabla M_{k}^{(j_{k})}\left(\BFX_{k}\right)$ denote the model gradient during the $k_j$th iteration of the contraction loop. Then $\mu\left\| \nabla M_{k}^{(j_{k})}\left(\BFX_{k}\right)\right\| <\Deltatilde_{k}w^{j_{k}-1},\ \forall j_k\geq1$. This means, since $w < 1$, that $\Deltatilde_{k}w^{j_{k}-1}\to0$ and therefore $\left\| \nabla M_{k}^{(j_{k})}\left(\BFX_{k}\right)\right\| \xrightarrow{wp1} 0$ as $j_k\to\infty$. Furthermore, due to the sampling rule in (\ref{eq:inner-sampling-interpolation}) and  by Theorem \ref{thm:chow-robbins}, we have that $N\left(\BFY_{i}^{(j_{k})}\right)\to\infty$ as $j_{k} \to \infty$. Now, if $E_{k,i}^{(j_{k})} = \Fbar\left(\BFY_{i}^{(j_{k})},N\left(\BFY_{i}^{(j_{k})}\right)\right) - f\left(\BFY_{i}^{(j_{k})}\right)$, then we can write for large enough $k$ and some $\delta > 0$, 
%\begin{align} \label{alg2term} \mbP\left\{ \frac{\sum_{i=2}^{p}\left|E_{k,i}^{(j_{k})}-E_{k,1}^{(j_{k})}\right|}{\Deltatilde_{k}w^{j_{k}-1}}\geq c\right\} & \leq \sum_{i=2}^{p}\mathbb{E}\left[\mbP\left\{ \left|E_{k,i}^{(j_{k})}-E_{k,1}^{(j_{k})}\right|\geq c\Deltatilde_{k}w^{j_{k}-1} \, \vert \, \mcF_k\right\}\right] \nonumber  \\
%& \leq 2(p-1)c^{-2}\left(\Deltatilde_{k}w^{j_{k}-1}\right)^{-2}(1+\delta)\kappa^2_{ias}\Delta_k^4\lambda_k^{-1} \nonumber \\
%& \leq 2(p-1)c^{-2}(1+\delta)\kappa^2_{ias}\Delta_k^2\lambda_k^{-1},
%\end{align} 
\begin{equation}\label{alg2term} \mbP\left\{ \frac{\sum_{i=2}^{p}\left|E_{k,i}^{(j_{k})}-E_{k,1}^{(j_{k})}\right|}{\Deltatilde_{k}w^{j_{k}-1}}\geq c\right\}\leq8(p-1)^{3}c^{-2}(1+\delta)\kappa_{ias}^{2}\Delta_{k}^{2}\lambda_{k}^{-1},\end{equation}
which follows from arguments identical to (\ref{eq:mod_err_bd}) in the proof of Lemma \ref{lem:G-convergence}.
Since the right-hand side of (\ref{alg2term}) is summable, we conclude by Borel-Cantelli's first lemma (Lemma \ref{lem:borel_cantelli}) that $\left(\Deltatilde_{k}w^{j_{k}-1}\right)^{-1}\sum_{i=2}^{p}\left|E_{k,i}^{(j_{k})}-E_{k,1}^{(j_{k})}\right| \xrightarrow{wp1}0$. This implies, from Lemma \ref{lem:model_error} and since Algorithm \ref{alg:sublisting} maintains full-linearity, that as $j_{k}\to\infty$, \[\left\| \nabla f\left(\BFX_{k}\right)-\nabla M_{k}^{(j_{k})}\left(\BFX_{k}\right)\right\| \leq\kappa_{1}\left(\Deltatilde_{k}w^{j_{k}-1}\right)^{\theta}+\kappa_{2}\frac{\sum_{i=1}^{p}\left|E_{k,i}^{(j_{k})}-E_{k,1}^{(j_{k})}\right|}{\left(\Deltatilde_{k}w^{j_{k}-1}\right)}\xrightarrow{wp1}0.\]
%Note that following the similar steps in \cite[p. 200]{conschvic2009} the full-linearity of the model is maintained in $\mcB(\BFX_{k},\Delta_{k})$ in its stochastic sense although the trust-region radius retrieved at the end of Algorithm \ref{alg:sublisting} can be larger than $$ in which the poised set is selected. 
Hence we have arrived at a contradiction since we argued that $\left\| \nabla M_{k}^{(j_{k})}\left(\BFX_{k}\right)\right\| \xrightarrow{wp1} 0$ but then $\|\nabla f(\BFX_{k})\| = c' \neq 0$ by the contrapositive assumption.

\end{proof}

\bibliographystyle{plain}
\bibliography{all,stochastic_optimization}

\begin{thebibliography}{10}

\bibitem{alagoz1}
O.~Alagoz, A.~J. Schaefer, and M.~S. Roberts.
\newblock {Optimization in Organ Allocation}.
\newblock In P.~Pardalos and E.~Romeijn, editors, {\em Handbook of Optimization
  in Medicine}. Kluwer Academic Publishers, 2009.

\bibitem{amos2014algorithm}
B.~D. Amos, D.~R. Easterling, L.~T. Watson, W.~I. Thacker, B.~S. Castle, and
  M.~W. Trosset.
\newblock {Algorithm XXX: QNSTOP Ñ Quasi Newton Algorithm for Stochastic
  Optimization}.
\newblock 2014.

\bibitem{anknelsta2010}
B.~Ankenman, B.~L. Nelson, and J.~Staum.
\newblock {Stochastic Kriging for Simulation Metamodeling}.
\newblock {\em Operations research}, 58(2):371--382, 2010.

\bibitem{asmussen2007stochastic}
S.~Asmussen and P.~W. Glynn.
\newblock {\em {Stochastic Simulation: Algorithms and Analysis}}, volume~57.
\newblock Springer Science \& Business Media, 2007.

\bibitem{bandeira2014convergence}
A.~S. Bandeira, K.~Scheinberg, and L.~N. Vicente.
\newblock {Convergence of Trust-Region Methods Based on Probabilistic Models}.
\newblock {\em SIAM Journal on Optimization}, 24(3):1238--1264, 2014.

\bibitem{bastin1}
F.~Bastin, C.~Cirillo, and P.~L. Toint.
\newblock {Convergence Theory for Nonconvex Stochastic Programming with an
  Application to Mixed Logit}.
\newblock {\em Mathematical Programming}, 108:207--234, 2006.

\bibitem{bastin2006adaptive}
F.~Bastin, C.~Cirillo, and Ph.~L. Toint.
\newblock {An Adaptive Monte Carlo Algorithm for Computing Mixed Logit
  Estimators}.
\newblock {\em Computational Management Science}, 3(1):55--79, 2006.

\bibitem{baymor07}
G.~Bayraksan and D.~P. Morton.
\newblock {Assessing Solution Quality in Stochastic Programs}.
\newblock {\em Mathematical Programming Series B}, 108:495--514, 2007.

\bibitem{baymor09}
G.~Bayraksan and D.~P. Morton.
\newblock {A Sequential Sampling Procedure for Stochastic Programming}.
\newblock {\em Operations Research}, 59(4):898--913, 2011.

\bibitem{baypie12}
G.~Bayraksan and P.~Pierre-Louis.
\newblock {Fixed-Width Sequential Stopping Rules for a Class of Stochastic
  Programs}.
\newblock {\em SIAM Journal on Optimization}, 22(4):1518--1548, 2012.

\bibitem{bazaara1}
M.~S. Bazaara, H.~Sherali, and C.~M. Shetty.
\newblock {\em Nonlinear Programming: Theory and Algorithms}.
\newblock John Wiley \& Sons, New York, NY., 2006.

\bibitem{bil95}
P.~Billingsley.
\newblock {\em Probability and Measure}.
\newblock Wiley, New York, NY., 1995.

\bibitem{billargra2013}
S.~C. Billups, J.~Larson, and P.~Graf.
\newblock {Derivative-Free Optimization of Expensive Functions with
  Computational Error Using Weighted Regression}.
\newblock {\em SIAM Journal on Optimization}, 23(1):27--53, 2013.

\bibitem{billups2013derivative}
S.~C. Billups, J.~Larson, and P.~Graf.
\newblock {Derivative-Free Optimization of Expensive Functions with
  Computational Error Using Weighted Regression}.
\newblock {\em SIAM Journal on Optimization}, 23(1):27--53, 2013.

\bibitem{blum1}
J.~Blum.
\newblock Approximation methods which converge with probability one.
\newblock {\em Annals of Mathematical Statistics}, 25(2):382--386, 1954.

\bibitem{broadie3}
M.~Broadie, D.~M. Cicek, and A.~Zeevi.
\newblock {An Adaptive Multidimensional Version of the Kiefer-Wolfowitz
  Stochastic Approximation Algorithm}.
\newblock In M.~D. Rossetti, R.~R. Hill, B.~Johansson, A.~Dunkin, and R.~G.
  Ingalls, editors, {\em Proceedings of the 2009 Winter Simulation Conference},
  pages 601--612. Institute of Electrical and Electronics Engineers:
  Piscataway, New Jersey, 2009.

\bibitem{broadie2}
M.~Broadie, D.~M. Cicek, and A.~Zeevi.
\newblock {General Bounds and Finite-Time Improvement for the Kiefer-Wolfowitz
  Stochastic Approximation Algorithm}.
\newblock {\em Operations Research}, 59(5):1211--1224, 2011.

\bibitem{chang2013stochastic}
K.~Chang, L.~J. Hong, and H.~Wan.
\newblock {Stochastic Trust-Region Response-Surface Method (STRONG) - a New
  Response-Surface Framework for Simulation Optimization}.
\newblock {\em INFORMS Journal on Computing}, 25(2):230--243, 2013.

\bibitem{chemensch2015}
R.~Chen, M.~Menickelly, and K.~Scheinberg.
\newblock {Stochastic Optimization Using a Trust-Region Method and Random
  Models}, 2015.
\newblock submitted.

\bibitem{chorob1965}
T.~S. Chow and H.~Robbins.
\newblock {On the Asymptotic Theory of Fixed-Width Sequential Confidence
  Intervals for the Mean}.
\newblock {\em The Annals of Mathematical Statistics}, pages 457--462, 1965.

\bibitem{conn2009global}
A.~R. Conn, K.~Scheinberg, and L.~N. Vicente.
\newblock {Global Convergence of General Derivative-Free trust-Region
  Algorithms to First-and Second-Order Critical Points}.
\newblock {\em SIAM Journal on Optimization}, 20(1):387--415, 2009.

\bibitem{conschvic2009}
A.~R. Conn, K.~Scheinberg, and L.~N. Vicente.
\newblock {\em {Introduction to Derivative-Free Optimization}}, volume~8.
\newblock Siam, 2009.

\bibitem{deng2006adaptation}
G.~Deng and M.~C. Ferris.
\newblock {Adaptation of the {UOBYQA} Algorithm for Noisy Functions}.
\newblock In {\em Proceedings of the 38th conference on Winter simulation},
  pages 312--319. Winter Simulation Conference, 2006.

\bibitem{deng2009variable1}
G.~Deng and M.~C. Ferris.
\newblock {Variable-Number Sample-Path Optimization}.
\newblock {\em Mathematical Programming}, 117(1-2):81--109, 2009.

\bibitem{durrett1}
R.~Durrett.
\newblock {\em Probability: Theory and Examples}.
\newblock Cambridge University Press, New York, NY, 2010.

\bibitem{garjudvic2015}
R.~Garmanjani, D.~J{\'u}dice, and L.~N. Vicente.
\newblock {Trust-Region Methods Without Using Derivatives: Worst Case
  Complexity and the Non-Smooth Case}.
\newblock 2015.

\bibitem{ghomuk1979}
M.~Ghosh and N.~Mukhopadhyay.
\newblock {Sequential Point Estimation of the Mean When the Distribution is
  Unspecified}.
\newblock {\em Communications in Statistics - Theory and Methods}, pages
  637--652, 1979.

\bibitem{ghomuksen97}
M.~Ghosh, N.~Mukhopadhyay, and P.~K. Sen.
\newblock {\em {Sequential Estimation}}.
\newblock Wiley Series in Probability and Statistics, 1997.

\bibitem{hashemi2014adaptive}
F.~S. Hashemi, S.~Ghosh, and R.~Pasupathy.
\newblock {On Adaptive Sampling Rules for Stochastic Recursions}.
\newblock In {\em Proceedings of the 2014 Winter Simulation Conference}, pages
  3959--3970. IEEE Press, 2014.

\bibitem{shane3}
S.~G. Henderson and B.~L. Nelson, editors.
\newblock volume~13 of {\em Handbooks in Operations Research and Management
  Science: Simulation}.
\newblock Elsevier, 2006.

\bibitem{hou2014applied}
Y.~T. Hou, Y.~Shi, and H.~D. Sherali.
\newblock {\em {Applied Optimization Methods for Wireless Networks}}.
\newblock Cambridge University Press, 2014.

\bibitem{kw1}
J.~Kiefer and J.~Wolfowitz.
\newblock {Stochastic Estimation of the Maximum of a Regression Function}.
\newblock {\em Annals of Mathematical Statistics}, 23:462--466, 1952.

\bibitem{kimpashen14}
S.~Kim, R.~Pasupathy, and S.~G. Henderson.
\newblock {A Guide to SAA}.
\newblock Frederick HillierÕs OR Series. Elsevier, 2014.

\bibitem{kushner2}
H.~J. Kushner and G.~G. Yin.
\newblock {\em Stochastic Approximation and Recursive Algorithms and
  Applications}.
\newblock Springer-Verlag, New York, NY., 2003.

\bibitem{larbil2014}
J.~Larson and S.~C. Billups.
\newblock {Stochastic Derivative-Free Optimization Using a Trust-Region
  Framework}.
\newblock {\em Under review at Computational Optimization and Applications},
  2014.

\bibitem{lar2012}
Jeffrey~M. Larson.
\newblock {\em Derivative-Free Optimization of Noisy Functions}.
\newblock PhD thesis, Department of Applied Mathematics, University of
  Colorado, Denver, CO, 2012.

\bibitem{makmorwoo99}
W.~K. Mak, D.~P. Morton, and R.~K. Wood.
\newblock {Monte Carlo Bounding Techniques for Determining Solution Quality in
  Stochastic Programs}.
\newblock {\em Operations Research Letters}, 24:47--56, 1999.

\bibitem{mokpel11}
A.~Mokkadem and M.~Pelletier.
\newblock {A Generalization of the Averaging Procedure: The Use of
  Two-Time-Scale Algorithms}.
\newblock {\em SIAM Journal on Control and Optimization}, 49:1523, 2011.

\bibitem{nocwri2006}
J.~Nocedal and S.~J. Wright.
\newblock {\em Numerical Optimization}.
\newblock Springer-Verlag, Berlin, 2006.

\bibitem{nsoesie2013simulation}
E.~O. Nsoesie, R.~J. Beckman, S.~Shashaani, K.~S. Nagaraj, and M.~V. Marathe.
\newblock {A Simulation Optimization Approach to Epidemic Forecasting}.
\newblock {\em PloS one}, 8(6):e67164, 2013.

\bibitem{ortrhe1970}
J.~M. Ortega and W.~C. Rheinboldt.
\newblock {\em Iterative Solution of Nonlinear Equations in Several Variables}.
\newblock Academic Press, New York, NY., 1970.

\bibitem{osorio2009surrogate}
C.~Osorio and M.~Bierlaire.
\newblock {A Surrogate Model for Traffic Optimization of Congested Networks: an
  Analytic Queueing Network Approach}.
\newblock {\em Report TRANSP-OR}, 90825:1--23, 2009.

\bibitem{pas10}
R.~Pasupathy.
\newblock {On Choosing Parameters in Retrospective-Approximation Algorithms for
  Stochastic Root Finding and Simulation Optimization}.
\newblock {\em Operations Research}, 58:889--901, 2010.

\bibitem{pasgho13}
R.~Pasupathy and S.~Ghosh.
\newblock {Simulation Optimization: A Concise Overview and Implementation
  Guide}.
\newblock INFORMS TutORials. INFORMS, 2013.

\bibitem{pasglyghohas14}
R.~Pasupathy, P.~W. Glynn, S.~G. Ghosh, and F.~S. Hashemi.
\newblock {How Much to Sample in Simulation-Based Stochastic Recursions?}
\newblock 2014.
\newblock Under Review.

\bibitem{pasupathy2006testbed}
R.~Pasupathy and S.~G. Henderson.
\newblock {A Testbed of Simulation-Optimization Problems}.
\newblock In {\em Proceedings of the 38th conference on Winter simulation},
  pages 255--263. Winter Simulation Conference, 2006.

\bibitem{pasupathy_testbed2011}
R.~Pasupathy and S.~G. Henderson.
\newblock {{SimOpt}: A Library of Simulation Optimization Problems}.
\newblock In S.~Jain, R.~R. Creasey, J.~Himmelspach, K.~P. White, and M.~Fu,
  editors, {\em Proceedings of the 2011 Winter Simulation Conference}.
  Institute of Electrical and Electronics Engineers: Piscataway, New Jersey,
  2011.

\bibitem{polyjudi92}
B.~T. Polyak and A.~B. Juditsky.
\newblock {Acceleration of Stochastic Approximation by Averaging}.
\newblock {\em SIAM Journal on Control and Optimization}, 30(4):838--855, 1992.

\bibitem{powell2002uobyqa}
M.~J.~D. Powell.
\newblock {{UOBYQA}: Unconstrained Optimization by Quadratic Approximation}.
\newblock {\em Mathematical Programming}, 92(3):555--582, 2002.

\bibitem{robbins1}
H.~Robbins and S.~Monro.
\newblock {A Stochastic Approximation Method}.
\newblock {\em Annals of Mathematical Statistics}, 22:400--407, 1951.

\bibitem{roysze_or2011}
J.~Royset and R.~Szechtman.
\newblock {Optimal Budget Allocation for Sample Average Approximation}.
\newblock {\em Operations Research}, 2011.
\newblock Under Review.

\bibitem{rusz1}
A.~Ruszczynski and A.~Shapiro, editors.
\newblock {\em Stochastic Programming. Handbook in Operations Research and
  Management Science}.
\newblock Elsevier, New York, NY., 2003.

\bibitem{shadenrus09}
A.~Shapiro, D.~Dentcheva, and A.~Ruszczynski.
\newblock {\em Lectures on Stochastic Programming: Modeling and Theory}.
\newblock SIAM, Philadelphia, PA, 2009.

\bibitem{spall3}
J.~C. Spall.
\newblock {Adaptive Stochastic Approximation by the Simultaneous Perturbation
  Method}.
\newblock {\em IEEE Transactions on Automatic Control}, 45:1839--1853, 2000.

\bibitem{spall4}
J.~C. Spall.
\newblock {\em {Introduction to Stochastic Search and Optimization}}.
\newblock John Wiley \& Sons, Inc., Hoboken, NJ., 2003.

\bibitem{stein1}
M.~L. Stein.
\newblock {\em Interpolation of Spatial Data: Some Theory for Kriging}.
\newblock Springer, New York, NY, 1999.

\bibitem{yousefian2012stochastic}
F.~Yousefian, A~Nedi{\'c}, U.~V., and Shanbhag.
\newblock {On Stochastic Gradient and Subgradient Methods with Adaptive Step
  Length Sequences}.
\newblock {\em Automatica}, 48(1):56--67, 2012.

\end{thebibliography}
\end{document}